\documentclass[11,reqno,oneside]{amsart}
\usepackage[margin=3cm]{geometry}
\usepackage{amsmath}
\usepackage{amssymb}
\usepackage{amsbsy}
\usepackage{amsfonts}
\usepackage{dsfont}
\usepackage[utf8]{inputenc}
\usepackage{enumerate}

\usepackage[ruled,vlined]{algorithm2e}
\usepackage[dvips]{epsfig}
\usepackage{colortbl}
\usepackage{algcompatible}
\usepackage{hyperref}
\usepackage{tabularx}
\usepackage{afterpage}
\usepackage{lscape}
\usepackage{multirow}
\usepackage{multicol}
\usepackage{arydshln}
\usepackage{cancel}
\usepackage{makecell}
\usepackage{booktabs}

\usepackage{amsfonts,amscd}
\usepackage[T1]{fontenc}
\usepackage{lmodern}
\usepackage{nicefrac}
\usepackage{graphics,caption}
\usepackage[labelfont=normalfont,textfont=normalfont]{subcaption}

\usepackage{arydshln}
\usepackage{lipsum}
\usepackage{amsfonts}
\usepackage{graphicx}
\usepackage{epstopdf}
\usepackage{mathtools}
\usepackage{autonum}
\usepackage[numbers,sort&compress]{natbib}
\usepackage{stmaryrd,xparse}
\usepackage{doi}

\usepackage{booktabs}

\include{definitions}
\graphicspath{{Figures/}}

\newtheorem{theorem}{Theorem}[section]

\def\V{\mathbb{V}}

\def\Vc{\bar{\mathbb{V}}}
\def\Vd{\mathbb{V}}

\def\Vl{\mathbb{V}_\omega}

\def\Vbddc{\tilde{\mathbb{V}}}

\def\Eq#1{(\ref{eq:#1})}

\def\TBD#1{{\color{red} [*#1*]}}

\def\M{{\mathcal{M}}}
\def\A{{\mathcal{A}}}

\def\sbx{\omega}

\def\Vlx{{{V}_{\sbx}}}
\def\Vx{{\V}}
\def\Al{{\A_\sbx}}

\def\vlx{{v_\sbx}}
\def\Ad{{{\mathcal{A}}}}
\def\Ac{{\bar{\mathcal{A}}}}
\def\Q{{\mathcal{Q}}}
\def\E{\mathcal{E}}
\def\Wx{{\mathcal{W}}}
\def\x{\boldsymbol{x}}

\begin{document}

\title[Solvers for unfitted finite element methods]{Robust and scalable domain decomposition solvers for unfitted finite element methods}

\thanks{This work has partially been funded by the European Research Council under the Proof of Concept Grant No. 737439 NuWaSim - On a Nuclear Waste Deep Repository Simulator and the EC - International Cooperation in Aeronautics with China (Horizon 2020) under the project: Efficient   Manufacturing for Aerospace Components USing Additive Manufacturing, Net Shape HIP and Investment Casting (EMUSIC). SB gratefully acknowledges the support received from the Catalan Government through the ICREA Acad\`emia Research Program.}

\author[S. Badia]{ Santiago Badia$^\dag$}

\thanks{$\dag$ Universitat Polit\`ecnica de Catalunya, Jordi Girona 1-3, Edifici C1, E-08034 Barcelona $\&$ Centre Internacional de M\`etodes Num\`erics en Enginyeria, Parc Mediterrani de la Tecnologia, Esteve Terrades 5, E-08860 Castelldefels, Spain E-mail: {\tt sbadia@cimne.upc.edu}.  }

\author[Francesc Verdugo]{Francesc Verdugo$^\ddag$}
\thanks{$\ddag$ Universitat Polit\`ecnica de Catalunya, Jordi Girona 1-3, Edifici C1, E-08034 Barcelona $\&$ Centre Internacional de M\`etodes Num\`erics en Enginyeria, Parc Mediterrani de la Tecnologia, Esteve Terrades 5, E-08860 Castelldefels, Spain E-mail: {\tt fverdugo@cimne.upc.edu}.}

\date{\today}

\begin{abstract}
 Unfitted finite element methods, e.g., extended finite element techniques or the so-called finite cell method, have a great potential for large scale simulations, since they avoid the generation of body-fitted meshes and the use of graph partitioning techniques, two main bottlenecks for problems with non-trivial geometries. However, the linear systems that arise from these discretizations can be much more ill-conditioned, due to the so-called small cut cell problem. The state-of-the-art approach is to rely on sparse direct methods, which have quadratic complexity and are thus not well suited for large scale simulations. In order to solve this situation, in this work we investigate the use of domain decomposition preconditioners (balancing domain decomposition by constraints) for unfitted methods. We observe that a straightforward application of these preconditioners to the unfitted case has a very poor behavior. As a result, we propose a {customization} of the classical BDDC methods based on {the stiffness weighting} operator and an improved definition of the coarse degrees of freedom in the definition of the preconditioner. These changes lead to a robust and algorithmically scalable solver able to deal with unfitted grids.   A complete set of complex 3D numerical experiments show the good performance of the proposed preconditioners.
\end{abstract}

\maketitle

%\noindent{\bf 2010 Mathematics Subject Classification:} 35Q30; 65N30; 76N10.


\tableofcontents

% OUTLINE OF THE ARTICLE

\section{Introduction}

% SANTI / FRANCESC

The use of unfitted finite element methods (FEMs) is an appealing
approach for different reasons.  { They}  are interesting in 
coupled problems that involve interfaces (e.g., fluid-structure
interaction \cite{badia_fluidstructure_2008} or free surface flows), or  situations in which one wants to avoid the generation
of body-fitted meshes. These type of techniques have been named in
different ways. Unfitted FEMs for capturing interfaces are usually
denoted as eXtended FEM (XFEM) \cite{belytschko_arbitrary_2001},
whereas these techniques are usually denoted as embedded (or immersed)
boundary methods, when the motivation is to simulate a problem using a
(usually simple Cartesian) background mesh. Recently, different
realizations of the method have been coined in different ways,
depending on the way the numerical integration is performed, how
Dirichlet boundary conditions are enforced on no-matching surfaces, or
the type of stabilization, if any, being used. To mention some
examples, the finite cell method combines an XFEM-type functional
space, a numerical integration based on { adaptive} cell refinement and full
sub-cell integration, to { perform} integration on cut cells, and a Nitsche
weak imposition of boundary conditions
\cite{parvizian_finite_2007}. CutFEM makes also use of Nitsche's
method, but includes additional ``ghost penalty" stabilization terms
to improve the  the stability of the algorithm
\cite{burman_cutfem:_2015}. The huge success of isogeometrical
analysis methods (spline-based discretization) and the severe
limitations of this approach in complex 3D geometries will probably
increase the interest of  { unfitted} methods in the near future
\cite{kamensky_immersogeometric_2015}.

The main showstopper up to now for the succesful application of
unfitted methods for realistic applications is the linear solver
step. The condition number of the resulting linear system does not
only depend on the characteristic size of the background mesh elements,
but also on the characteristic size of the cut elements; cut elements
can be arbitrarily small and can have arbitrarily high aspect ratios
for small cut cell situations. Enforcing some reasonable thresholds,
one can use robust sparse direct solvers for these problems. However,
sparse direct methods are very expensive, due to their quadratic
complexity, which is especially dramatic at large scales. Scalability
is also hard to get, especially at very large scales. For large scale
applications with body-fitted meshes in CSE, iterative solvers,
usually Krylov solvers combined with preconditioners, is the natural
way to go. Scalable preconditioners at large scales are usually based
on algebraic multigrid (AMG) and domain decomposition solvers
\cite{rudi_extreme-scale_2015,gmeiner_quantitative_2016,badia_multilevel_2016,klawonn_toward_2015}. Unfortunately,
the lack of robust and scalable iterative solvers for unfitted FEM  has limited the applicability of unfitted
methods in real applications. It is well-known that some rudimentary
methods, e.g., Jacobi preconditioners, can solve the issue of the
ill-conditioning that comes from cut elements. However, these
preconditioners are neither scalable nor optimal. Different serial
linear solvers for unfitted FEMs have been recently proposed (see,
e.g.,
\cite{menk_robust_2011,berger-vergiat_inexact_2012,hiriyur_quasi-algebraic_2012,de_prenter_condition_2017}).
The methods in \cite{menk_robust_2011,berger-vergiat_inexact_2012}
consider a segregation of nodes into healthy and ill nodes, and
a domain decomposition solver for the ill nodes. The method
is only applied to 2D problems in \cite{menk_robust_2011} and serial
computations, and the domain decomposition solver is not scalable, due
to the fact that no coarse correction is proposed. The method in
\cite{berger-vergiat_inexact_2012} proposes to use an AMG method for
the healthy nodes. A specific-purpose serial AMG solver is designed in
\cite{hiriyur_quasi-algebraic_2012} and a serial incomplete
factorization solver is considered in
\cite{de_prenter_condition_2017}. 

%Other recent
%% approaches involve tailored incomplete LU factorizations \TBD{}, even
%% though these methods are mainly useful for serial computations, since
%% they do not exhibit large amounts of concurrency. Works on parallel
%% hybrid multigrid additive Schwarz solvers for unfitted FE methods have
%% been proposed in \TBD{}. However, the approach presented therein is
%% limited by the fact that \emph{all} ill-posed DoFs are sent to a
%% coarse problem, which is too large in realistic 3D applications. (The
%% application in mind in \TBD{} is fracture mechanics, under the
%% assumption that the number of fractures is small.)

Motivated by the lack of robust (with respect to the element cuts) and
scalable parallel solvers for unfitted methods, we develop in this
work a domain decomposition solver with these desired properties based
on the balancing domain decomposition by constraints (BDDC)
framework. First, we provide a short description of the mathematical
analysis that proves the scalability of BDDC methods on body-fitted
meshes. Next, we show an example that proves that the use of standard
BDDC methods cannot be robust with respect to the element cuts. In
fact, the method performs poorly when straighforwardly
applied to unfitted meshes. Next, we propose (based on heuristic
arguments and numerical experimentation, but motivated from the
mathematical body of domain decomposition methods) a modified BDDC
method. First, we consider {stiffness weighting operator (already
  proposed in the original BDDC article
  \cite{dohrmann_preconditioner_2003})}, relying on the diagonal of
the (sub-assembled stiffness matrix). The use of this weighting
operators proves to be essential for the robustness of the
solver. Next, we consider enhanced coarse spaces, with a sub-partition
of the edge constraints. Since no mathematical analysis is available
for domain decomposition methods with stiffness weighting, we have
performed a comprehensive set of numerical experiments on 3D
geometries with different levels of complexity. The methods are not
only robust and algorithmically weakly scalable, but what is more
surprising, the number of iterations seems to be independent of the
domain shape too. Furthermore, the cost of the modified weighted
operator is identical to the standard one, and the algorithm has the
same building blocks as the standard BDDC preconditioner. This is
good, since one can use the extremely scalable implementation of these
methods, e.g., in the \texttt{FEMPAR} library
\cite{fempar-manual,fempar} or in the PETSc library \cite{zampini_pcbddc:_2016}. In any case, the cost of the new method
can differ from the standard one in the number of coarse degrees of
freedom (DOFs). We have observed that in practice, the CPU cost of the
modified formulation is very close to the one of the standard
preconditioner. In fact, as we increase the size of the coarse problem
of the modified solver (which is robust for unfitted methods) tends to
the one of the standard preconditioner (which performs very bad for
unfitted meshes). We note that the preconditioner proposed herein
could also be readily applied to XFEM-enriched interface problems and
other discretization techniques (e.g., discontinuous Galerkin
methods).

Let us describe the outline of this work. In Sect. \ref{section2} we
show  { the unfitted FE method considered in this paper, in particular}  our choice to integrate  { in} cut cells and  the surface
intersection { algorithm}. (In any case, the preconditioners proposed later on do
not depend on these choices, and can be used in other situations). We
also state the model problem, its discretization, and the mesh and
subdomain partitions. The standard BDDC preconditioner and its
building blocks for body-fitted meshes are presented in
Sect. \ref{section3}. In Sect. \ref{sec:bddc-unfitted}, we show a breakdown example of the
preconditioner on unfitted meshes, an expensive solution that is
provably robust, and a cheap solution based on the stiffness weighting
and (optionally) two slightly larger coarse spaces. A complete set of { complex 3D}
numerical experiments are included in Sect. \ref{sec:examples} to show
experimentally the good performance of the proposed preconditioner. Finally, some conclusions are drawn and future work is described in Sect. \ref{conclusions}.

\section{Unfitted FE method}\label{section2}

%In this section, we describe the immersed boundary method leading to the system of linear equations to be solved with the proposed domain decomposition methods.
%The method is almost the same as the finite cell method. Similarities: Unfitted Cartesian meshes and Nitsche's Method for imposing Dirichlet boundary conditions. Differences: We consider an integration strategy based on marching cubes and the level set method, where as the Finite Cell Method uses an adaptive structured integration cells and exact geometry representation. We could adaopt also their integration strategy without any change in the solver.

\subsection{Cell partition and subdomain partition}
\def\domart{\Omega_{\rm art}}
\def\domext{\Omega_{\rm ext}}
\def\domin{\Omega_{\rm in}}
\def\domcut{\Omega_{\rm cut}}

\def\pelart{\theta_{\rm art}}
\def\pelext{\theta_{\rm ext}}
\def\pelin{\theta_{\rm in}}
\def\pelcut{\theta_{\rm cut}}

Let $\Omega\subset\mathbb{R}^d$ be an open bounded {polygonal} domain, with
$d\in\{2,3\}$ the number of spatial dimensions. For the sake of
simplicity and without loss of generality, we consider that the domain
boundary is defined as the zero level-set of a given scalar function
$\phi^\mathrm{ls}$, namely
$\partial\Omega\doteq\{x\in\mathbb{R}^d:\phi^\mathrm{ls}(x)=0\}$. We
note that the problem geometry could be described using 3D CAD data
instead of level-set functions, by providing techniques to compute the
intersection between cell edges and surfaces (see, e.g.,
\cite{marco_exact_2015-1}). In any case, the way the geometry is
handled does not affect the domain decomposition solver presented
below. Like in any other immersed boundary method, we build the
computational mesh by introducing an \emph{artificial} domain
$\domart$ such that it has a simple geometry easy to mesh using
Cartesian grids and it includes the \emph{physical} domain $\Omega
\subset\domart$ (see Fig.~\ref{fig:IBM:a}).
\begin{figure}[ht!]
\centering
\begin{subfigure}[b]{0.31\textwidth}
\includegraphics[width=\textwidth]{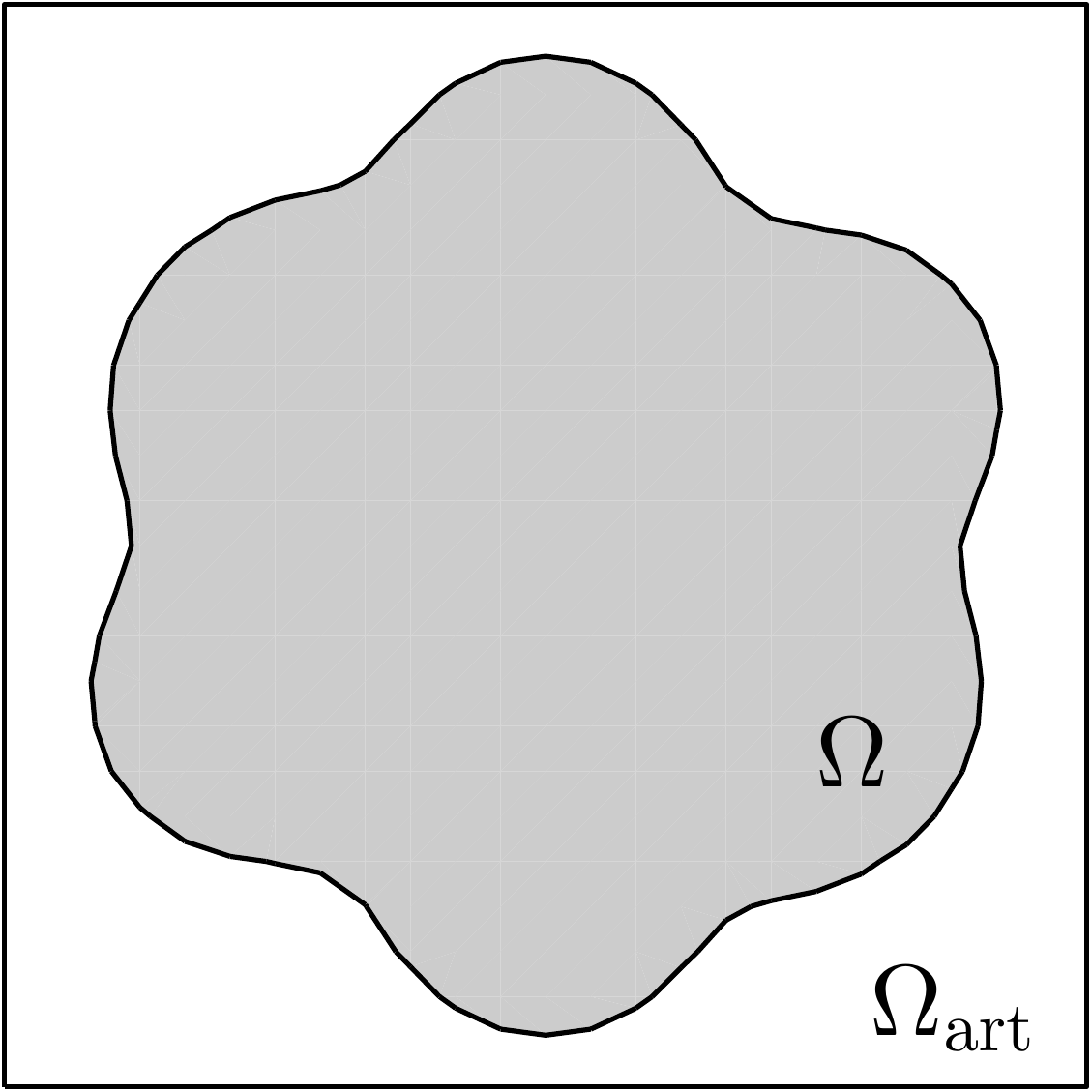}
\caption{Geometry}
\label{fig:IBM:a}
\end{subfigure}
\hspace{0.3cm}
\begin{subfigure}[b]{0.31\textwidth}
\centering

\begin{small}
{\color[RGB]{6,169,193} \rule{8pt}{8pt}} External
{\color[RGB]{124,191,123} \rule{8pt}{8pt}} Internal
{\color[RGB]{240,185,73} \rule{8pt}{8pt}} Cut
\end{small}

\vspace{0.1cm}

\includegraphics[width=\textwidth]{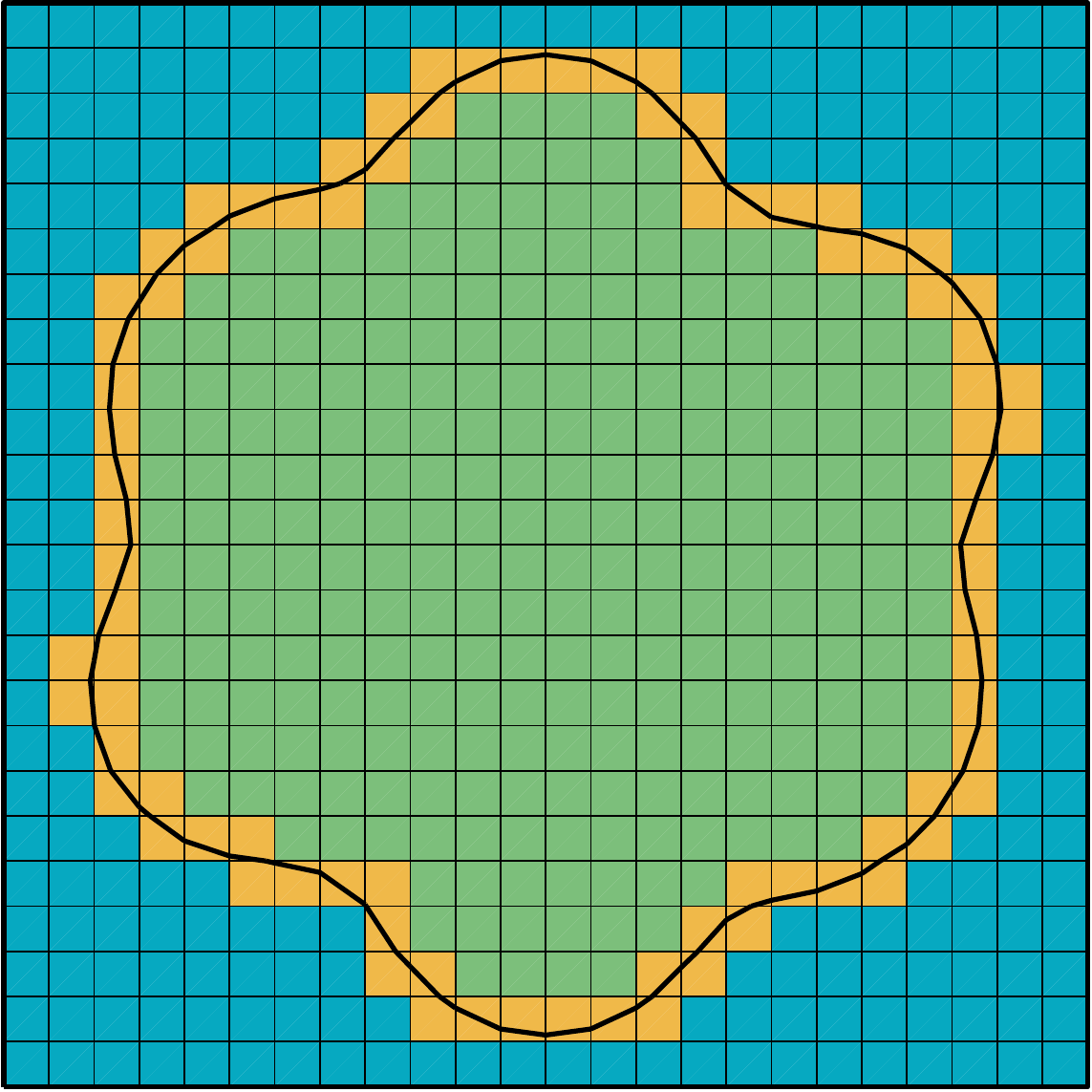}
\caption{Background mesh.}
\label{fig:IBM:b}
\end{subfigure}

\vspace{0.5cm}

\begin{subfigure}[b]{0.31\textwidth}
\includegraphics[width=\textwidth]{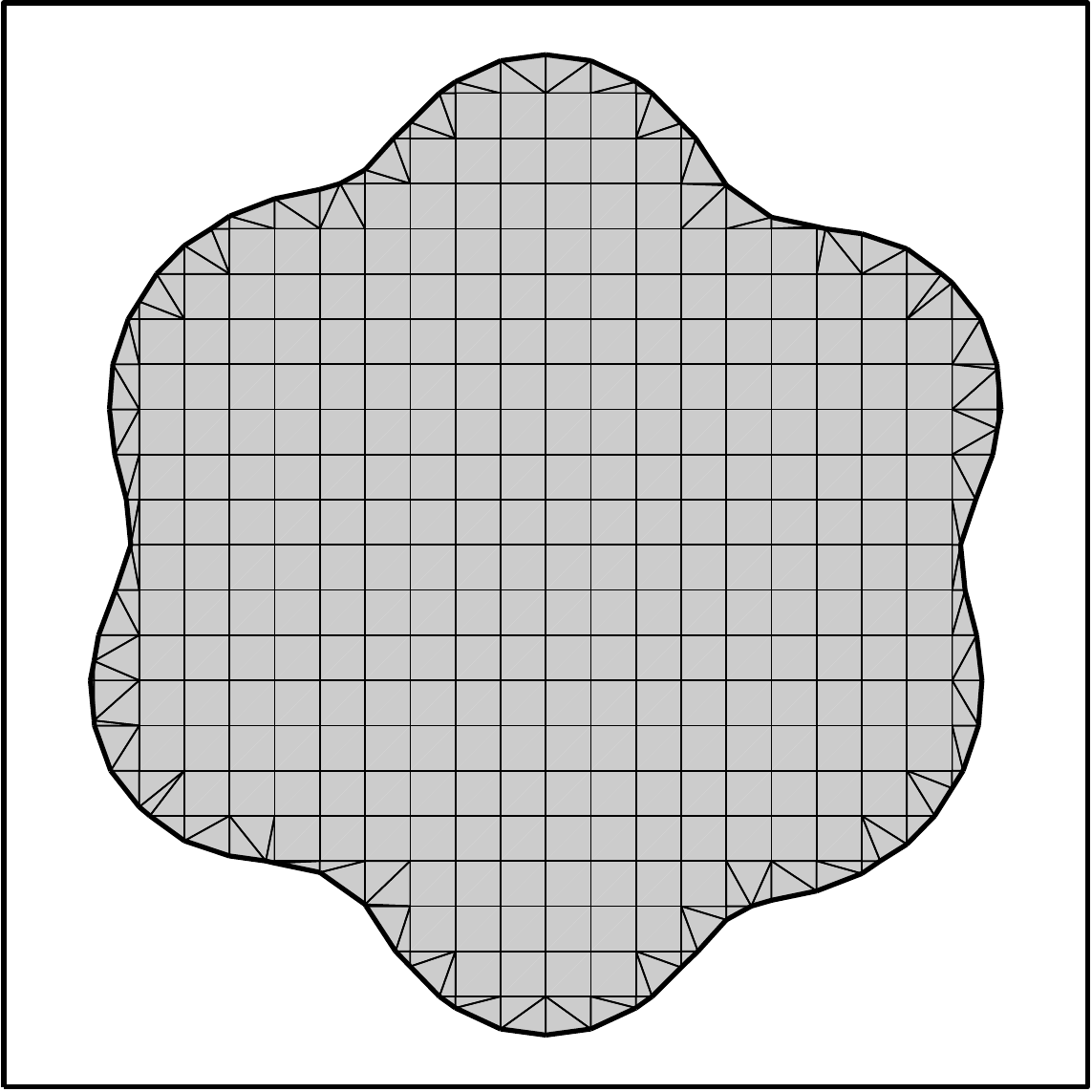}
\caption{Integration cells.}
\label{fig:IBM:c}
\end{subfigure}
\hspace{0.3cm}
\begin{subfigure}[b]{0.31\textwidth}
\includegraphics[width=\textwidth]{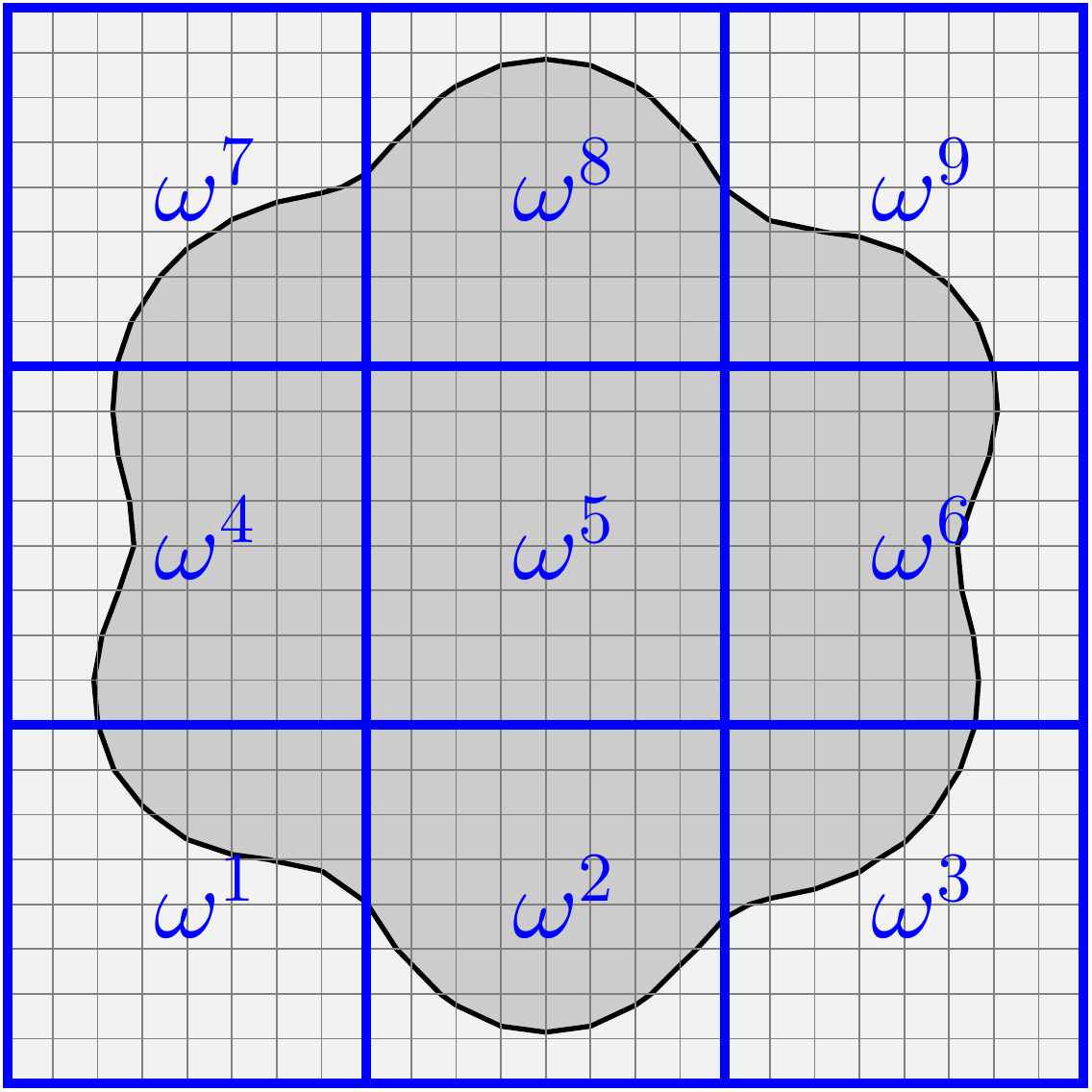}
\caption{Subdomains.}
\label{fig:IBM:d}
\end{subfigure}

\caption{Immersed boundary setup.}
\label{fig:IBM}
\end{figure} 

\def\domart{\Omega_{\rm art}}
\def\domext{\Omega_{\rm ext}}
\def\domin{\Omega_{\rm in}}
\def\domcut{\Omega_{\rm cut}}
\def\domact{\Omega_{\rm act}}

\def\pel{\theta}
\def\pelact{\theta_{\rm act}}
\def\pelart{\theta_{\rm art}}
\def\pelext{\theta_{\rm ext}}
\def\pelin{\theta_{\rm in}}
\def\pelcut{\theta_{\rm cut}}

\def\pelactloc{\theta_{\rm act}^\omega}

\def\psbact{\Theta}%_{\rm act}}

\def\J{\mathcal{J}}

Let us construct a partition of $\domart$ into \emph{cells},
represented by $\pelart$. We are interested in $\pelart$ being a
Cartesian mesh into hexahedra for $d=3$ or quadrilaterals for $d = 2$,
even though unstructured background meshes can also be
considered. Cells in $\pelart$ can be classified as follows: a cell $e
\in \pelart$ such that $e \subset \Omega$ is an \emph{internal cell};
if $e \cap \Omega = \emptyset$, $e$ is an \emph{external cell};
otherwise, $e$ is a \emph{cut cell} (see Fig.~\ref{fig:IBM:b}). The
set of interior (resp. external and cut) cells is represented with
$\pelin$ and its union $\domin \subset \Omega$
(resp. $(\pelext,\domext)$ and $(\pelcut, \domcut))$. Furthermore, we
define the set of \emph{active cells} as $\pelact \doteq \pelin \cup
\pelcut$ and its union $\domact$. Let us also consider a subdomain
partition $\psbact$ of $\domact$, obtained by \emph{aggregation of
  active cells} in $\pelact$, i.e., there is an element partition
$\pelactloc \doteq \{ e \in \pelact : e \subset \omega \}$ for any
$\omega\in\Theta$.  The interface of the subdomain partition is
$\Gamma \doteq \cup_{\omega \in \Theta} \partial \omega \setminus
\partial \Omega$.

To generate the subdomain partition in the context of immersed
boundary methods is simple, since Cartesian meshes are allowed. For
uniform Cartesian meshes, the elements can be easily aggregated into
uniform subdomains as shown in Fig.~\ref{fig:IBM:d}, whereas for
adaptive Cartesian meshes the elements can be efficiently aggregated
using space-filling curves \citep{bader_space-filling_2012}. In both
cases, the partitions can be generated efficiently in parallel without
using graph-based partitioning techniques
\cite{karypis_metis_2011}, which are otherwise required for complex
unstructured body-fitted meshes. In fact, the elimination of both
body-fitted mesh generation and graph partitioning is the main
motivation of this work.

\def\ie{k}
\def\elt{{\delta_\ie}} 
      
\subsection{Model problem and space discretization}\label{sec:modpr} Let us consider as a model problem the Poisson equation with Dirichlet and Neumann boundary conditions: find $u \in H^1(\Omega)$ such that 
\begin{equation}
%\left\lbrace
%\begin{array}{rl}
-\Delta u = f  \quad \text{in } \ \Omega, \qquad 
u=g^\mathrm{D}  \quad\text{on } \ \Gamma_\mathrm{D}, \qquad
\nabla u \cdot n = g^\mathrm{N}  \quad \text{on } \ \Gamma_\mathrm{N},
%\end{array}
%\right.
\label{eq:PoissonEq}
\end{equation}
where $(\Gamma_\mathrm{D},\Gamma_\mathrm{N})$ is a partition of the domain boundary (the Dirichlet and Neumann boundaries, respectively), $f\in H^{-1}(\Omega)$, $g^\mathrm{D}\in H^{1/2}(\Gamma_\mathrm{D})$, and $g^\mathrm{N}\in H^{-1/2}(\Gamma_\mathrm{N})$. 

\newcommand{\gradient}{\boldsymbol{\nabla}}
\newcommand{\normal}{\boldsymbol{n}}

For the space discretization, we consider $H^1$-conforming FE spaces
on conforming meshes.
%(The discontinuous Galerkin (DG) case will not
%be considered in this work, but we could also develop preconditioned
%iterative linear solvers for unfitted DG formulations by combining the
%ideas herein with the BDDC preconditioner for DG formulations in
%\cite{dryja_bddc_2007}.)
 For unfitted grids, it is not possible to
include Dirichlet conditions in the approximation space in a strong
manner. Thus, we consider Nitsche's method
\cite{becker_mesh_2002,nitsche_uber_1971} to impose Dirichlet boundary
conditions weakly on $\Gamma_D$. It provides a consistent numerical
scheme with optimal converge rates (also for high-order elements) that
is commonly used in the immersed boundary community
\cite{schillinger_finite_2014}. We define the space $\Vc \subset
H^1(\Omega)$ as the global FE space related to the mesh
$\pelact$. Further, we define the FE-wise operators:
\begin{align}
  & \A_e(u,v) \doteq  \int_{e \cap \Omega} \gradient u \cdot \gradient v \mathrm{\ d}V + \int_{\Gamma_\mathrm{D} \cap e}\left( \beta u v  - v \left(\normal\cdot \gradient u\right) -  u \left(\normal\cdot \gradient v\right) \right)  \mathrm{\ d}{S},   \\
& \ell_e(v) \doteq {\int_{\Gamma_\mathrm{D} \cap e}\left( \beta  vg^\mathrm{D}  -  \left(\normal\cdot \gradient v\right)g^\mathrm{D} \right)  \mathrm{\ d}S}, & 
\end{align}
    {defined for a generic mesh element $e\in\pelact$}. Vector
    $\normal$ denotes the outwards normal to $\partial\Omega$. The
    bilinear form $\A_e(\cdot,\cdot)$ includes the usual form
    resulting for the integration by parts of \eqref{eq:PoissonEq} and
    the additional term associated with the weak imposition of
    Dirichlet boundary conditions with Nitsche's method. The
    right-hand side operator $\ell_e(\cdot)$ includes additional terms
    related to Nitsche's method. We will make abuse of notation, using
    the same symbol for a bilinear form, e.g., $\A
    : \V \rightarrow \V'$, and its corresponding linear operator, i.e., $\langle \A u , v \rangle \doteq \A( u
    , v)$.

    The coefficient $\beta>0$ is a mesh-dependent parameter that has
    to be large enough to ensure the coercivity of
    $\A_e(\cdot,\cdot)$.
%We compute the minimal values of $\beta$ required to  ensure coercivity with local element-wise eigenvalue problems, as proposed in \cite{de_prenter_condition_2017}. 
{ As shown in \cite{de_prenter_condition_2017}, coercivity is mathematically guaranteed by considering an element-wise constant coefficient $\beta$ such that
\begin{equation}
\beta_e \geq C_e \doteq \sup_{v\in \Vc} \dfrac{ \mathcal{B}_e (v,v)  }{ \mathcal{D}_e(v,v) }
%\label{eq:coer-cond}
\end{equation}
for all the mesh elements $e\in\pelcut$ intersecting the boundary $\Gamma^\mathrm{D}$. In previous formula, $\beta_e$ is the value of $\beta$ restricted to element $e$, whereas  $\mathcal{D}_e(\cdot,\cdot)$ and  $\mathcal{B}_e(\cdot,\cdot)$ are the forms defined as
\begin{equation}
\mathcal{D}_e(u,v)\doteq \int_{e \cap \Omega} \gradient u \cdot \gradient v \mathrm{\ d}V,\quad \text{ and }\quad
\mathcal{B}_e (u,v) \doteq  \int_{\Gamma_\mathrm{D} \cap e} \left(\normal\cdot \gradient u\right) \left(\normal\cdot \gradient v\right) \mathrm{\ d}S.
\end{equation}
Since $\Vc$ is finite dimensional, and $\mathcal{D}_e(\cdot,\cdot)$ and  $\mathcal{B}_e(\cdot,\cdot)$ are symmetric and bilinear forms, the value  $C_e$ (i.e., the minimum admissible coefficient $\beta_e$)
can be computed numerically as $C_e\doteq\tilde\lambda_\mathrm{max}$, being $\tilde\lambda_\mathrm{max}$  the largest eigenvalue of the generalized eigenvalue problem  (see \cite{de_prenter_condition_2017} for details):
\begin{equation}
\mathcal{B}_e x = \tilde\lambda \mathcal{D}_e x.
\label{eq:local-eigenvals}
\end{equation}
In this work, we chose $\beta_e\doteq 2\tilde\lambda_\mathrm{max}$ in all the examples in order to be on the safe side, as suggested in~\cite{de_prenter_condition_2017}. Let us note that both operators have the same kernel, the space of constants. Thus, the generalized eigenvalue problem \eqref{eq:local-eigenvals} is well-posed (see~\cite{de_prenter_condition_2017} for more details). 

The global FE operator $\Ac: \Vc \rightarrow \Vc'$.  and right-hand side term $\ell \in \Vc' $ are stated as the sum of the element contributions, i.e., $$
\Ac(u,v) \doteq \sum_{e \in \pelact} \A_e(u,v),  \quad
\ell{ (v)} \doteq \sum_{e \in \pelact} \ell_e{ (v)},
\qquad  \hbox{for } u, v \in \Vc. 
$$

Further, we define $b: \Vc' \rightarrow \Vc$ as $b(v) \doteq f(v) + g^{\rm N}(v) + \ell(v)$, for $v \in \Vc$.

The global problem can be stated in operator form as: find $u \in \Vc$ such that $\Ac u = b$ in $\Vc'$.  After the definition of the set of DOFs $\bar{\mathcal{N}}$ and the corresponding FE basis (of shape functions) $\{\phi_a(\x)\}_{a \in \bar{\mathcal{N}}}$ that span $\Vc$, the previous problem leads to a linear system to be solved. At large scales, sparse linear systems are usually solved with Krylov iterative methods \cite{saad_iterative_2003} due to the quadratic complexity of sparse direct methods. Here, we can consider a conjugate gradient (CG) iterative method since $\Ac$ is symmetric positive definite. The convergence rate of the conjugate gradient method is very sensitive to the condition number of the operator, namely $k_2(\Ac)\doteq|\Ac|_2/|\Ac^{-1}|_2$, where $|\Ac|_2$ denotes the operator 2-norm. As shown in \cite{de_prenter_condition_2017} (under over-optimistic assumptions), the condition number scales as  
\begin{equation}
k_2(\Ac) \sim  \min_{e \in \pelcut} \eta_e^{-(2p_e+1-2/d)}, \qquad \eta_e \doteq  \frac{| e \cap \Omega|}{ |e|},
\label{eq:finite-cell-estimate}
\end{equation}
where $p_e$ is the polynomial degree of the interpolation at element $e$. Thus, arbitrarily high condition numbers are expected in practice since the position of the interface cannot be controlled.  In turn, the convergence rate of the iterative solver is expected to be very slow and highly sensitive to the position of the geometry unless a robust preconditioner is considered. Our goal is to develop a BDDC preconditioner able to deal with the bad conditioning properties associated with cut elements, be quasi-optimal with respect to $h$ (the characteristic element size of $\pelact$), and be weakly scalable (not dependent on the number of subdomains for a fixed subdomain problem size). Attaining these objectives, we will have at our disposal a solver for embedded boundary systems  that will efficient exploit distributed memory resources for large scale problems.

\section{Domain decomposition for body-fitted FE meshes}\label{section3}

In this section, we briefly introduce the BDDC method as usually considered in the literature for conventional body-fitted meshes \cite{dohrmann_preconditioner_2003}. The description of the algorithm is concise and we mainly focus on the parts that are modified later in Sect.~\ref{sec:bddc-unfitted} to deal with unfitted meshes. We refer the reader to \cite{mandel_convergence_2003,brenner_mathematical_2010} for a detailed exposition of BDDC methods and to~\cite{badia_implementation_2013,badia_multilevel_2016} for the practical implementation details. We have also included a short proof of the condition numbers for body-fitted meshes. It is important for the subsequent discussion to explain the breakdown of the standard algorithm for unifitted meshes and motivate robust modifications. In this section, $\Omega \equiv \Omega_{\rm act}$ and the Dirichlet boundary conditions are strongly enforced in the definition of the FE spaces.
%In the following, define the basic ingredients of the BDDC method (the sub-asse, the weighting operator and the coarse degrees of freedom), and finally we state the preconditioner.

%Since in this section we restrict (for the moment) to body fitted meshes, we assume that the mesh~$\mathcal{T}$ fits to the physical domain $\Omega$, or in other words, no exterior or cut elements are found in $\mathcal{T}$. 

\subsection{Sub-assembled problem}

Non-overlapping domain decomposition preconditioners rely on the definition of a sub-assembled FE problem, in which contributions between subdomains have not been assembled. In order to do so, at every subdomain $\omega \in \Theta$, we consider the FE space $\Vd_\omega$ associated to the element partition $\pelact^\omega$ with homogeneous {Dirichlet boundary conditions on $\partial\omega \cap \partial \Omega_{\rm D}$}. One can define the subdomain operator $\Al(u,v) = \sum_{e \in \pelact^\omega} \A_e(u,v)$, for $u, v \in \Vl$.

Subdomain spaces lead to the sub-assembled space of functions $\V \doteq \Pi_{\omega \in \Theta} \V_\omega$. For any $u \in \Vd$, {we define its restriction to a subdomain $\omega \in \Theta$ as $u_\omega$}. Any function $u \in \Vd$ can be represented by  its unique decomposition into subdomain functions as $\{u_\omega \in \Vlx\}_{\omega \in \Theta}$. We also define the sub-assembled operator $\mathcal{A}(u,v) \doteq \Pi_{\omega \in \Theta} {\mathcal{A}}_\omega (u_\sbx,v_\sbx)$. 
%At this point, we can state the following sub-assembled problem: given $g \in \V_\omega'$, find $u \in \V_\omega$ such that $\Al \ulx = g_\sbx$ for $\sbx \in \Theta$. 
%\begin{align}\label{eq:sub-assembled_space}
%& \Al \ulx = g_\sbx, \quad \sbx \in \Theta.
%\end{align}
%With the previous notation, we can write the sub-assembled problem in a compact manner as $\Ad u = g$. %\P

With these definitions,  $\Vc$ can be understood as the subspace of functions in $\Vd$ that are continuous on the interface $\Gamma$, and $\bar{\mathcal{A}}$ is the Galerkin projection of $\mathcal{A}$ onto $\Vc$. We note that $\pelact$ and the FE type defines $\Vc$, whereas $\Theta$ is also required to define the local  spaces $\{\V_\omega\}_{\omega \in \Theta}$ and the sub-assembled space $\Vd$, respectively. We consider Lagrangian FE spaces, where DOFs are associated to nodes (spatial points). We represent the set of nodes related to the FE space $\V$ with $\mathcal{N}$ (analogusly for $\mathcal{N}_{\omega}$).

\subsection{Coarse DOFs}\label{sec:std-objects}

A key ingredient in domain decomposition preconditioners is to classify the set of nodes of the FE space $\Vc$. The interface $\partial e$ of every FE in the mesh $\pelact$ can be decomposed into vertices, edges, and faces. By a simple classification of these entities, based on the set of subdomains that contain them, one can also split the interface $\Gamma$ into faces, edges, and vertices (at the subdomain level), that will be called geometrical objects. We represent the set of geometrical objects by $\Lambda$. In all cases, edges and faces are open sets in their corresponding dimension. By construction, faces belong to two subdomains and edges belong to more than two subdomains in three-dimensional problems. This classification of $\Omega$ into objects automatically leads to a partition of interface DOFs into DOF objects, due to the fact that every DOF in a FE does belong to only one geometrical entity.  These definitions are heavily used in domain decomposition preconditioners (see, e.g., \cite[p.~88]{toselli_domain_2004}).

Next, we associate to some (or all) of these geometrical objects a \emph{coarse} DOF. In BDDC methods, we usually take as coarse DOFs mean values on a subset of objects $\Lambda_O$. Typical choices of $\Lambda_O$ are $\Lambda_O \doteq \Lambda_{C}$, when only corners are considered, $\Lambda_O \doteq \Lambda_C \cup \Lambda_{E}$, when corners and edges are considered, or $\Lambda_O \doteq \Lambda$, when corners, edges, and faces are considered. These choices lead to three common variants of the BDDC method referred as BDDC(c), BDDC(ce) and BDDC(cef), respectively. This classification of DOFs into objects can be restricted to any subdomain $\omega \in \Theta$, leading to the set of subdomain objects $\Lambda_O(\omega)$.

\def\rest{\mathcal{R}}
With the classification of the interface nodes and the choice of the objects in $\Lambda_O$, we can define the coarse DOFs and the corresponding BDDC space. Given an object $\lambda \in \Lambda_O(\omega)$, let us define its restriction operator $\rest_\lambda^\omega$ on a function $u \in \Vl$ as follows: $\rest_\lambda^\omega(u)(\xi) = u(\xi)$ {for  } a node $\xi \in \mathcal{N}_{\omega}$ that belongs to the geometrical object $\lambda$, and zero otherwise. We define the BDDC space $\Vbddc \subset \Vx$ as the subspace of functions $v \in \Vx$ such that the constraint
\begin{equation}\label{eq:coarse-dofs}
\alpha_\lambda^\sbx(\vlx) \doteq  \int_{\lambda} \rest_\lambda^\sbx(\vlx)  \qquad \hbox{is identical for all } \sbx \in {\rm neigh}(\lambda){,}  
\end{equation}
    {where ${\rm neigh}(\lambda)$ stands for the set of subdomains that contain the object $\lambda$.} (The integral on $\lambda$ is just the value at the vertex, when $\lambda$ is a vertex.) Thus, every $\lambda \in \Lambda_O$ defines a coarse DOF value \Eq{coarse-dofs} that is continuous among subdomains. Further, we can define the BDDC operator $\tilde \Ad$ as the Galerkin projection of $\Ad$ onto $\Vbddc$. %\P
    
\subsection{Transfer operator}\label{sec:topo-weighting}

The next step is to define a transfer operator from the sub-assembled space $\Vd$ to the continuous space $\Vc$. The transfer operator is the composition of a weighting operator and a harmonic extension operator.

\begin{enumerate}

\item The weighting operator $\mathcal{W}$ takes a function $u \in \Vd$ and computes mean values on interface nodes, i.e.,
\begin{equation}\label{eq:weighting}
\mathcal{W}u(\xi) \doteq \frac{\sum_{\omega \in {\rm neigh}(\xi)}u_\omega(\xi)}{|{\rm neigh}(\xi)|}, %\qquad \xi \in \Xi,
\end{equation} at every node $\xi \in \mathcal{N}$ of the FE mesh $\pelact$, where ${\rm neigh}(\xi)$ stands for the set of subdomains that contain the node $\xi$. It leads to a continuous function $\mathcal{W}u \in \Vc$. It is clear that this operator only modifies the DOFs on the interface. We denote this weighting as \emph{topological weighting}, {since it is defined from the topology of the mesh and the subdomain partition}. % Other choices can be defined for non-constant physical coefficients. 

\item Next, let us define the bubble space $\V_0 \doteq \{ v \in \V: v = 0 \hbox{ on } \Gamma \}$ %We proceed analogously to define $\Vc_{\omega,0}$ and $\Vc_0$; we note that $\V_0 = \Vc_0 \subset \Vc$ by definition. 
  and the Galerkin projection $\Ad_0$ of $\Ad$ onto $\Vx_0$. We also define the trivial injection $\mathcal{I}_0$ from $\V_0$ to $\Vc$. The \emph{harmonic} extension reads as $\mathcal{E} v \doteq (I - \mathcal{I}_0 \Ad_0^{-1} \mathcal{I}_0^T \Ac) v$ for a function $v\in\Vc$, where $I$ is the identity operator in $\Vc$. This operator corrects interior DOFs only. The computation of $\Ad_0^{-1}$ involves to solve local problems with homogeneous Dirichlet boundary conditions on $\Gamma$.  %Thus, the operator can be decomposed into subdomain-wise harmonic extension operators $\E_\omega$. 
  The set of discrete global harmonic functions is $\V_0^\perp\doteq\{v\in\V:\ \A(v,w)=0,\ \forall w\in\V_0\}$.   
\end{enumerate}
The transfer operator $\Q : \Vd \rightarrow \Vc$ is defined as $\Q \doteq \E \Wx$. %\P

\subsection{Preconditioner and condition number bounds}

With all these ingredients, we are now in position to define the BDDC preconditioner. This preconditioner is an additive Schwarz preconditioner (see, e.g., \cite[Ch.~2]{toselli_domain_2004}), with corrections in $\V_0$ and the BDDC correction in $\Vbddc$ with the transfer $\Q$. As a result, the BDDC preconditioner reads as:
{
  \begin{equation}\label{eq:bddc}
\M \doteq \mathcal{I}_0 \A_0^{-1} \mathcal{I}_0^T+ \Q \tilde \A^{-1} \Q^T.
  \end{equation}
  }
For body-fitted meshes and second order elliptic problems,  the BDDC preconditioned operator, namely $\M\Ac$, has a condition number bounded as follows (see \cite{mandel_convergence_2003}).

\begin{theorem}
The operator $\M \Ac$, where $\M$ can be the preconditioner BDDC(c) or BDDC(ce) in 2D and BDDC(ce) and BDDC(cef) in 3D, has a condition number bounded by:
\begin{equation}
 k_2(\M\Ac)\leq \tilde C \left(1+\mathrm{log}^2 \left(\dfrac{H}{h}\right)\right),
 \label{eq:cond-estim}
\end{equation}
where $\tilde C$ is a constant independent from the sizes $h$ and $H$ (the characteristic mesh and subdomain size). In the case of BDDC(ce), every face $\lambda_F \in \Lambda_F$ must have an edge $\lambda_E \in \Lambda_E$  on its boundary, i.e., $\lambda_E \subset \partial \lambda_F$.
\end{theorem}
\begin{proof}
\def\E{\mathcal{E}}
\def\W{\mathcal{W}}
The BDDC method can be casted in the abstract additive Schwarz setting (see \cite[Ch.~7]{brenner_mathematical_2010} and \cite[Ch.~2]{toselli_domain_2004}). As a result, the condition number of the preconditioned operator is $k_2(\M\Ac) = \frac{\lambda_{\rm max}}{\lambda_{\rm min}}$, with $\lambda_{\rm min} = 1$ and
\begin{align}
%\lambda_{\rm min} &\doteq \min_{v \in \Vc \setminus \{0\}} \frac{\langle \A v , v \rangle}{ \min_{\{ v_0 \in \V_0, \,\tilde v \in \Vbddc: v_0 + \mathcal{Q} \tilde{v} = v \} } \langle \A v_0, v_0 \rangle + \langle \A \tilde v, \tilde v \rangle}, \\
\lambda_{\rm max} &\doteq \max_{\tilde v \in \Vbddc} \frac{\langle \A \Q \tilde v,\Q \tilde v\rangle}{\langle \A \tilde v, \tilde v \rangle} 
\leq 1 + \max_{\tilde v \in \Vbddc \cap \V_0^\perp} \frac{ \langle \A (\Q \tilde v  - \tilde v),\Q \tilde v - \tilde v \rangle}{\langle \A \tilde v, \tilde v \rangle},
\end{align}
where we have used the triangle inequality and the energy minimization property of discrete harmonic functions in the last bound (see \cite{mandel_convergence_2003}). 
%It is straighforward to check that $\lambda_{\rm min} = 1$, taking $\tilde v = \mathcal{E}v$ and $v_0 = v - \mathcal{E} v$ for any $v \in \Vc$ and noting that $\mathcal{W}v = v$.
In order to bound $\lambda_{\rm max}$, we need some further elaboration. We can decompose any discrete harmonic function $v \in\V_0^\perp$ as the sum of object functions $v = \sum_{\lambda \in \Lambda} \sum_{\omega \in {\rm neigh}(\lambda)} \E \rest_\lambda^\omega(v)$. For body-fitted meshes, it is clear that
\begin{equation}
\rest_{\lambda}^\omega (v) = 0, \quad  \hbox{on} \quad \partial\omega \setminus \lambda,  \quad \forall \lambda \in \Lambda,\quad \forall \omega \in {\rm neigh}(\lambda) .\label{eq:restbou}
\end{equation}
Further, given $v\in \V_0^\perp$, $\lambda \in \Lambda$, and subdomains $\omega, \omega' \in {\rm neigh}(\lambda)$, we define the jump function $\J_{\omega,\omega'}^\lambda(v)\in \V_0^\perp$ that is equal to zero in all subdomains but $\omega$, where it take the value $\J_{\omega,\omega'}^\lambda(v) (\xi) \doteq  v_\omega(\xi) - v_{\omega'}(\xi)$ if $\xi \in \lambda$ and zero otherwise.  Thus, we get
\begin{align}
w  \doteq \Q v - v  = \sum_{\lambda \in \Lambda} \sum_{\omega \in {\rm neigh}(\lambda)} \E \rest_\lambda^\omega(\Q v - v) = \sum_{\lambda \in \Lambda} \sum_{\omega \in {\rm neigh}(\lambda)} \sum_{\omega'{\rm neigh}(\lambda)} \frac{1}{|{\rm neigh}(\lambda)|} \J_{\omega,\omega'}^\lambda (v).\label{eq:weicon}
\end{align}
% Let us consider a constant $\gamma$ such that
% \begin{equation}
% \| \J_{\omega,\omega'}^\lambda (v) \|_\A \leq \gamma \| v \|_\A, \quad  \hbox{for any} \quad v \in \Vbddc, \, \lambda \in \Lambda, \, \omega, \omega' \in {\rm neigh}(\lambda).
% \label{eq:beta}
% \end{equation}
Using standard domain decomposition theory, one can prove the following bound  (see, e.g., \cite{mandel_convergence_2003}):
$$
\| \J_{\omega,\omega'}^\lambda (v) \|_\A \lesssim \left(1 + \mathrm{log}^\rho \left(\dfrac{H}{h}\right) \right) | v |^2_{H^1(\omega)},
$$
where $\rho = 1$ for edges and 2 for faces.  The bound for faces in BDDC(cef) differs from the fact that $\alpha_\lambda^\omega(v) \neq \alpha_\lambda^{\omega'}(v)$ in general. It can be bounded similarly relying on the assumption that there exists an edge in $\partial \lambda$ that is in $\Lambda_O$ (see \cite[p. 182]{toselli_domain_2004} for more details).

\end{proof}
In a weak scaling test (i.e., refining the mesh and the partition at the same rate), $H/h$ is constant and so is the condition number bound regardless of the problem size.  This property leads to an optimal weak scaling of the preconditioner (i.e., to increase $h$ and $H$ in the same proportion) when used within a conjugate gradient iteration. That is, the number of iterations needed to solve the problem up to a certain (relative) tolerance is proportional to the condition number $k_2(\M\Ac)$,  and as the condition number is bounded by a constant independent of the problem size, so are also the number of iterations. The coarse DOFs can be reduced, including only a subset of edges in $\Lambda_O$. We refer to \cite{klawonn_dual-primal_2002}  for more details and the extension to linear elasticity problems.

%{Then, I will propose a Method A, for which I will prove that the method works but is expensive. Next, I will propose a method such that the ratio is bounded for the diagonal matrix, and thus, not theory. I will show that it is affected by jumps in the weighting, which will motivate the second Method B. Finally, I will propose a heuristic method C.}

%{Breakdown outline: The typical corner/edge/face restrictions are not restrictions anymore, additional energy terms, cannot be bounded. Put silly example and counterexample. Thus, say that cannot be robust with respect to cuts.}

%{Improvements: Do a partition to minimize cut elements on faces; Neumann boundary cut elements in touch with full elements, we can take the value from the full one, and it should work. Cut/cut elements on interface, as corners. Dirichlet cut nodes should be corners too. Not all as coarse DOFs. Further reduction by acceptable path.}

%{Method A: Put corners such that the problem is fixed. All interface nodes on the subdomain interface to the coarse space. Thus, theory like in domain decomposition, saying that the inverse inequality is with $h$. Expensive but d-2.}

%{Method B: with the previous corrections.}

%{Method B-C: Develop a method that bounds the norm of the weighting operator with respect to the matrix diagonal (Method C). Say that it can be considered with respect to orthogonal basis but expensive. (Method B). Finally, show the effect of the weighting variations, and fix a tolerance. Relations with deluxe scaling?.}

%{Method D: Say that big jumps when cutting edges, only these additional terms.}

\section{Domain decomposition solvers for unfitted FE meshes}
\label{sec:bddc-unfitted}

%SANTI

The BDDC method presented in the previous section can be applied verbatim to unfitted meshes. In that case, the computational mesh is $\pelact$ and the extended domain $\domact$. It obviously requires specific strategies for integrating the problem matrices within cut elements and to impose Dirichlet conditions on unfitted boundaries, but the BDDC solver can be applied unaltered. The question is, however, whether the performance of BDDC is strongly deteriorated or not when moving from body-fitted to unfitted meshes. %By looking at the condition number bound \eqref{eq:cond-estim}, it is clear that we might expect problems for unfitted grids. 

The bound for the condition number in \eqref{eq:cond-estim} cannot be extended to unfitted methods. One of the problems is that \eqref{eq:restbou} for the physical subdomain $\omega\cap\Omega$, namely
\begin{equation}
\rest_{\lambda}^\omega (v) = 0, \quad  \hbox{on} \quad \partial(\omega\cap\Omega) \setminus \lambda, \quad \forall \lambda \in \Lambda,\quad \forall \omega \in {\rm neigh}(\lambda),
\end{equation}
is not true for unfitted meshes, since a partition can have cut cells on the subdomain interface.  As a result, the mathematical theory of substructuring domain decomposition methods, which is grounded on the object decomposition of functions and trace theorems, cannot be readily applied. In fact, it is easy to find examples that prove that the standard BDDC method breaks down for unfitted meshes.

\subsection{Breakdown example} 
\label{sec:breackdown-ex}

{Assuming that  \eqref{eq:restbou} holds for the unfitted problem at hand (which cannot be enforced in practice), it is easy to check that the standard analysis of BDDC methods, relying on trace inequalities, holds.  We observe that $w$ in \eqref{eq:weicon} vanishes on corners for $v \in \Vbddc$, since corner values are identical on all subdomains. Thus, only edge and face objects must be considered. Let us consider edge terms. We consider $\lambda \in \Lambda_E$ and $\omega, \omega' \in {\rm neigh}(\lambda)$. Using the standard analysis of BDDC methods, we note that the mean value of any $v \in \Vbddc$ is continuous between subdomains for BDDC(ce) and BDDC(cef), i.e., $\alpha_\lambda^\omega(v) = \alpha_\lambda^{\omega'}(v)$. We have:
\begin{align}
\| \J_{\omega,\omega'}^\lambda (v) \|_\A & = 
| \J_{\omega,\omega'}^\lambda (v) |_{H^1(\omega)} = C 
| \J_{\omega,\omega'}^\lambda (v) |_{H^{\frac{1}{2}}(\partial \omega)} \\ & \leq
| \rest_\lambda^\omega(v) - \alpha_\lambda^\omega(v) |_{H^{\frac{1}{2}}(\partial \omega)} + 
| \rest_\lambda^{\omega'}(v) - \alpha_\lambda^{\omega'}(v) |_{H^{\frac{1}{2}}(\partial \omega')},
\label{bound-j} 
\end{align}
where we have used the trace theorem in $H^1(\omega)$, the fact that the mean value of functions in $\Vbddc$ are continuous between subdomains, the fact that $\J_{\omega,\omega'}^\lambda (v)$ is discrete harmonic, and \eqref{eq:restbou}. More specifically, \eqref{eq:restbou} is used to say that   $\J_{\omega,\omega'}^\lambda (v) = 0$ on $\partial\omega \setminus \lambda$. We proceed analogously for faces in BDDC(cef). However, if \eqref{eq:restbou} does not hold, \emph{additional terms would appear on the boundary, e.g., the ones related to Nitsche's method, in the first equality in \eqref{bound-j}, that cannot be bounded as above}.\footnote{We note that Nitsche's methods for unfitted FEMs can require arbitrarily large penalty methods also affected by the small cut cell problem.} Even for Neumann boundary conditions on the unfitted boundary only, bounds cannot be independent of cuts, as described with the following example.}
%
%The face terms in BDDC(ce) can be bounded as follows:
%\begin{align}
%\| \J_{\omega,\omega'}^\lambda (v) \|_\A & = 
%\| \J_{\omega,\omega'}^\lambda (v) \|_{H^1(\omega)} = C 
%\| \J_{\omega,\omega'}^\lambda (v) \|_{H^{\frac{1}{2}}(\partial \omega)} \\ & \leq
%\| \rest_\lambda^\omega(v) - \alpha_\lambda^\omega(v)\|_{H^{\frac{1}{2}}(\partial \omega)} + 
%\| \rest_\lambda^{\omega'}(v) - \alpha_\lambda^{\omega'}(v) \|_{H^{\frac{1}{2}}(\partial \omega')}. 
%\end{align}

Let us consider the problem in Fig.~\ref{fig:ex_rectangle} and the BDDC(c) preconditioner in 2D. The left side of the domain is at distance $\epsilon$ to the interface. The underlying PDE is a Poisson problem with Neumann boundary conditions on the left (moving) side and Dirichlet Boundary conditions on the rest of the boundary, defined on a 2D rectangular domain that is dicretized with an unfitted mesh of bilinear Lagrange elements and partitioned into eight uniform subdomains.  The example is designed in such a way that, as $\varepsilon$ tends to zero, the area of subdomains $\omega^1$ and {$\omega^5$} tend to zero (see Fig.~\ref{fig:ex_rectangle:geom}). {Now, we consider a function $v \in \Vbddc$ that vanishes everywhere but in {$\omega^1$}, where it has value 1 on all DOFs that belong to the edge shared by {$\omega^1$} and {$\omega^2$} and zero elsewhere. Thus, the energy of the function before weighting is $\int_{\omega^1} |\nabla v|^2 = C_1 \epsilon$, where $C_1$ is a constant independent from $h$. On the other hand, after communication, the energy related to $\omega^2$ is $\int_{\omega^2} |\nabla v|^2 \geq C_2$ for some $C_2$ independent of $\epsilon$ and $h$. As a result, the condition number is bounded by $\frac{C_2}{C_1 \epsilon}$.\footnote{We note that it could not happen if \eqref{eq:restbou} would hold (as in body-fitted methods), because the gradient would increase as $\frac{1}{\epsilon}$.}} As a result, 
\begin{equation}
\lambda_{\rm max} = \max_{\tilde v \in \Vbddc} \frac{\langle \A \Q \tilde v,\Q \tilde v\rangle}{\langle \A \tilde v, \tilde v \rangle} \label{eq:keybou}
\end{equation}
tends to $\infty$ as $\varepsilon \to 0$, i.e., the condition number of the BDDC preconditioner cannot be uniform with respect to the surface intersection. (Similar examples can be designed for 3D problems and the other variants of the preconditioner.) It is clearly observed that the usual topological weighting operator leads to a solver that is not robust with respect to the position of the interface. It is also observed experimentally in Fig.~\ref{fig:ex_rectangle} that the iteration count explodes as $\varepsilon$ tends to zero.

\begin{figure}[ht!]
\begin{subfigure}[b]{0.45\textwidth}
\includegraphics[width=\textwidth]{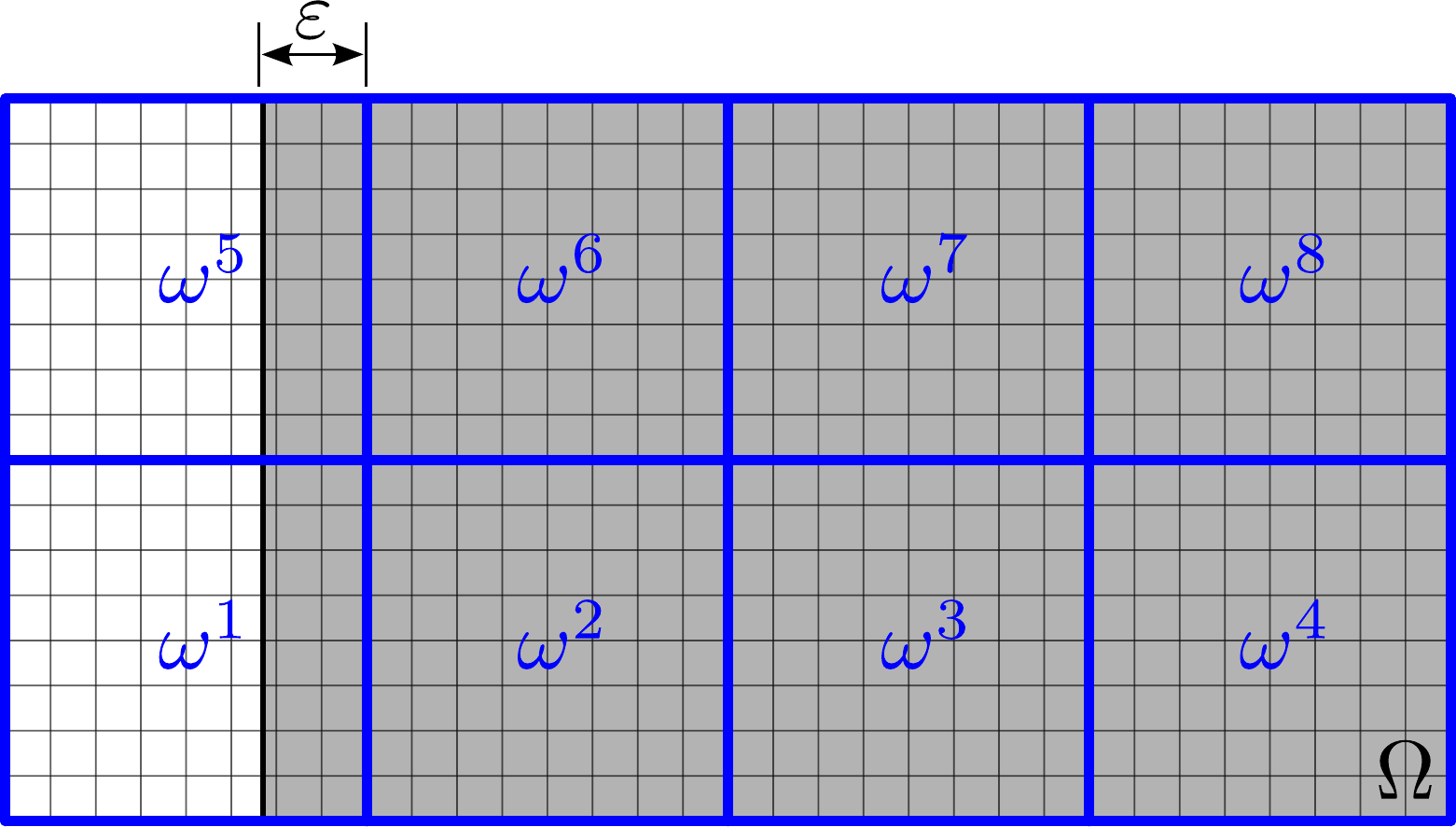}
\vspace{0.2cm}
%\begin{small}
%Homogeneous Neumann boundary
%Homogeneous Dirichlet boundary
%Non-homogeneous Neumann boundary
%\end{small}

\vspace{0.1cm}
\caption{}%Computational domain.}
\label{fig:ex_rectangle:geom}
\end{subfigure}
\hspace{0.3cm}
\begin{subfigure}[b]{0.5\textwidth}
\includegraphics[width=\textwidth]{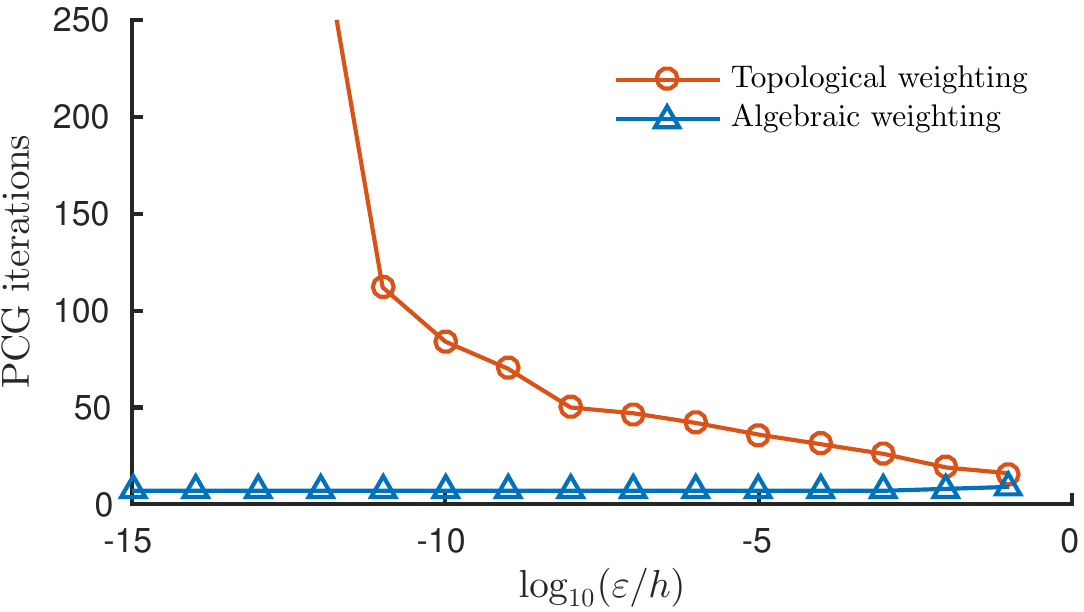}
\caption{}%Robustness w.r.t. the position of the interface.}
\label{fig:ex_rectangle:iters}
\end{subfigure}
\caption{Algebraic weighting vs topological weighting.}
\label{fig:ex_rectangle}
\end{figure}

\subsection{Enlarged coarse spaces and adaptive BDDC methods} 
In order to fix the BDDC preconditioner for unifitted FE meshes, one could consider as coarse corner DOFs all the nodes that violate \eqref{eq:restbou} and are on the interface. Such approach would lead to a huge space. The authors note that reduced coarse spaces could be considered, especially when the cut cells are related to Neumann boundary conditions. In any case, even though improved solutions can be designed, the coarse space is still too large for practical computations. In the following section, we consider a different approach related to stifness weighting.

Recently, adaptive BDDC methods have been developed based on the idea to use specral solvers to define the coarse space in such a way that the bound for \eqref{eq:keybou} can be controlled by the user (see, e.g., {\cite{pechstein2016unified,sousedik_adaptive-multilevel_2013-1,klawonn_feti-dp_2015,zampini_robustness_2016,kim_bddc_2016,calvo_adaptive_2016}}). The analysis of these methods does not rely on trace theorems (a difference with respect to classical BDDC methods). The analysis is algebraic (not relying on trace theorems, functional analysis tools, or FE inverse inequalities) and these algorithms could readily be used for unfitted FE methods to lead to provably robust solvers. However, such approach has two limitations: the additional cost related to the eigenvalue solvers is very large and it only pays the price in some cases, and it would lead to a severe load imbalance (only subdomains intersected by the surface would really require these eigenvalue solvers for many problems of interest). For these reasons, we do not pursue this line in this work.

\subsection{Stifness weighting operator}
\label{sec:new-weighting}

In this section, we will present a modification of the BDDC preconditioner that has been experimentally observed to be robust with respect to small cuts in unfitted FE methods with a moderate increase of the coarse solver size. The use of stiffness matrix-based weighting has previously been used in other situations \cite{dohrmann_preconditioner_2003} with good results, {e.g., to increase robustness of BDDC preconditioners  in isogeometric analysis \cite{beirao_da_veiga_bddc_2012}}. Unfortunately, there are no rigorous condition number bounds for this approach.

We can represent functions in $\Vc$ as vectors in $\mathbb{R}^{|\mathcal{N}|}$, and define the symmetric positive definite matrix $A_{ab} \doteq \langle \A \phi_b, \phi_a\rangle$ (analogously for the subassembled subdomain matrix $A_{ab}^\omega$). We introduce the so-called \emph{algebraic} weighting operator  $\mathcal{W}: \V \rightarrow \Vc$ for a given  function $u\in\V$ and a node $\xi$ of the mesh:  
\begin{equation}
\mathcal{W}u(\xi) \doteq \frac{\sum_{\omega \in {\rm neigh}(\xi)}A_{\xi \xi}^\omega u_\omega(\xi)}{\sum_{\omega \in {\rm neigh}(\xi)}A_{\xi \xi}^\omega}, %\qquad \xi \in \Xi.
\label{eq:alge-weighting}
\end{equation}
Formula \eqref{eq:alge-weighting} provides the node values after the weighting, which uniquely defines a FE function in $\Vc$. This type of weighting is referred to as the \emph{algebraic} weighting since it is defined from the problem matrices, in contrast to the topological weighting (see Sect.~\ref{sec:topo-weighting}) that is only based on the topology of the partition.

As we will show, the algebraic weighting leads to a much more robust BDDC method with respect to the relative position of domain boundary $\partial\Omega$ and the unfitted grid. We consider the 2D example in Fig.~\ref{fig:ex_rectangle}, for which the BDDC preconditioner with topological weighting is not robust. Fig.~\ref{fig:ex_rectangle:iters} shows the number of PGC iterations in function of the position of the interface. 
%The results are for two different types of weightings, the  \emph{topological} weighting previously presented in Sect. \TBD{} , and the \emph{algebraic} weighting to be presented below.  
It is clearly observed that the usual topological weighting operator leads to a solver that is not robust with respect to the position of the interface. The iteration count explodes as $\varepsilon$ tends to zero. On the other hand, the algebraic weighting operator leads to an extremely robust solver. The number of iterations are nearly independent of the position of the interface, even when the value $\varepsilon/h$ reaches machine precision.

Using the algebraic weighting, we have successfully solved also complex 3D examples (see the numerical results in Sect.~\ref{sec:examples}). However, we have observed that the number of iterations needed to solve the systems for meshes with cut elements are (about 2 times) higher that the iterations needed to solve the same problem without cutting the mesh elements. This increase in iterations is not dramatic, but further improvement is possible. To this end, we propose two further enhancements of the BDDC method based on the definition of the coarse objects that lead to almost identical iteration counts than the case in which all cut cells are filled and the standard BDDC method is used. 

\subsection{New coarse objects: variant 1}
\label{sec:new-objects-v1}

%2) Consider constraints on objects without cut elements that provide
%what is needed for a scalable method. We can do that for edges. The
%idea is that in the analysis for a given edge, the estimates blow up
%for cut edges.

%3) Increase the number of constraints, based on the variations of the
%weighting operator at one object. All the existing domain decomposition theory requires
%constant weighting coefficients, in order to prove estimates.

%Even though we have no complete mathematical theory for the
%preconditioner, its designed is based on the mathematical theory of domain decomposition methods.

The current mathematical analysis of BDDC methods is for body-fitted meshes, and therefore, it implicitly assumes that the coarse objects are not intersected by the boundary $\partial\Omega$. However, when body-fitted meshes are allowed, the objects can be cut by $\partial\Omega$ (see Fig.~\ref{fig:new-corners-1:a}). In this situation, the main mathematical properties required to prove scalability of the BDDC method are lost. It can lead to situations in which the corner constraints are exterior corners of cells with close to zero $\eta_e$, possibly leading to close to singular local problems. Motivated by this fact, we propose a modification in the definition of the coarse objects, which eliminates intersections with the boundary $\partial\Omega$. {Since \emph{fixing} the elements of the kernel (the constant for the Laplacian or rigid-body modes for elasticity)} for all the edges is enough in 3D (see, e.g., \cite{klawonn_dual-primal_2002}), we will only consider the modification of edge coarse DOFs.

\begin{figure}[ht!]

\centering

{\color{blue}$\bullet$} Nodes on coarse edges \hspace{0.5cm} {\color{red} $\boldsymbol{\times}$} New corners

\begin{subfigure}[b]{0.33\textwidth}
\includegraphics[width=\textwidth]{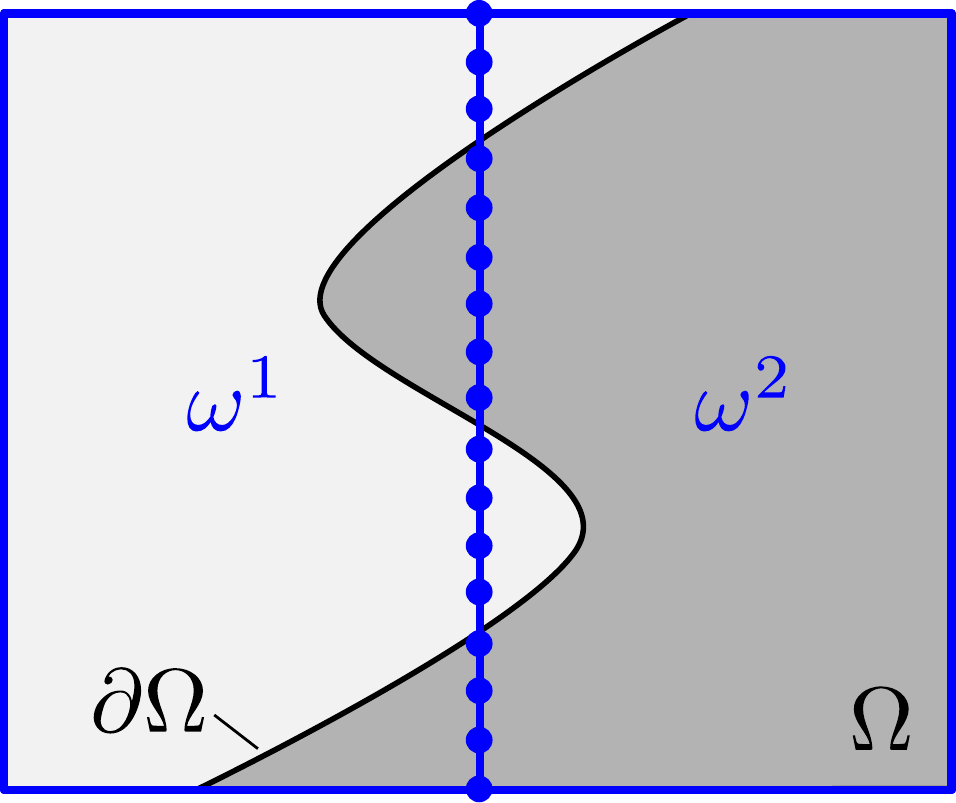}
\caption{Before}
\label{fig:new-corners-1:a}
\end{subfigure}
\hspace{0.5cm}
\begin{subfigure}[b]{0.33\textwidth}
\includegraphics[width=\textwidth]{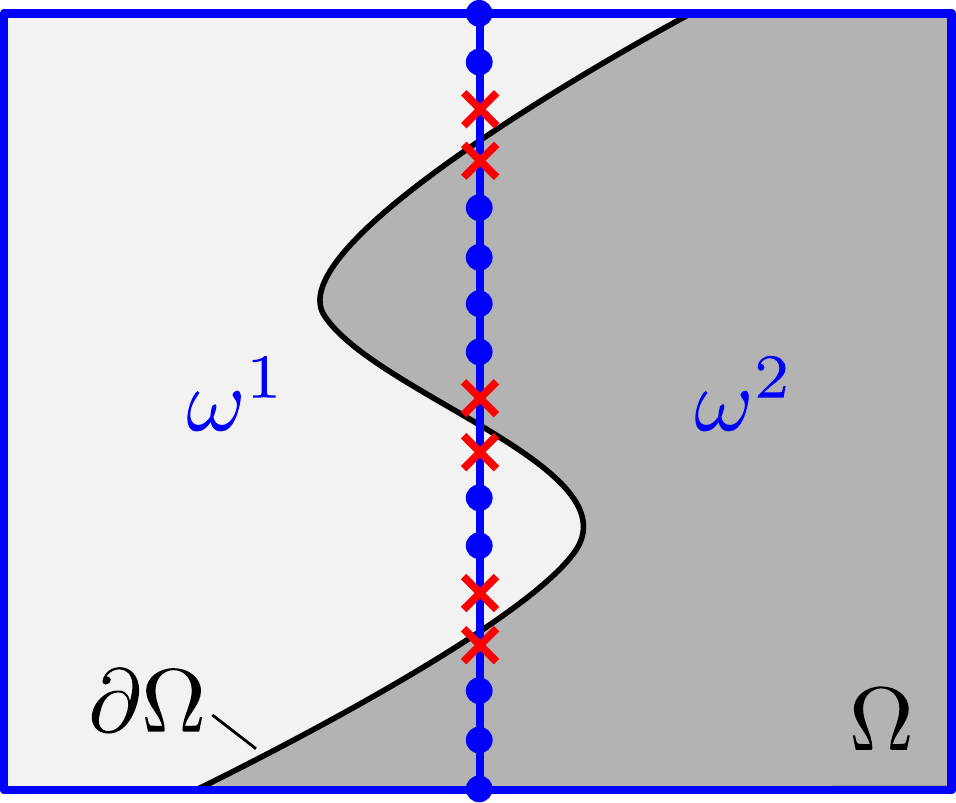}
\caption{After}
\label{fig:new-corners-1:b}
\end{subfigure}
\caption{Splitting a cut edge into several uncut edges (first variant).}
\label{fig:new-corners-1}
\end{figure}

The new objects are computed in two main steps. The first one is to build the standard objects $\Lambda_O$ of the BDDC method for $\pelact$, as presented in Sect.~\ref{sec:std-objects}. The second step is to identify the edges in $\Lambda_O$ cut by the boundary and split them into new edges/corners, leading to uncut edges (see Fig.~\ref{fig:new-corners-1}). We consider a segregation of a coarse edge $\lambda \in \Lambda_E$ as follows. Given the set of FE edges $t \subset \lambda$, we segregate them into interior edges $t \subset \Omega$, exterior edges such that $t \cap \Omega = \emptyset$, and the remaining cut edges. All DOFs on cut edges are considered as coarse corner DOFs. The rest of FE edges are agregated into \emph{connected} subedges (see Fig. \ref{fig:new-corners-1:b}).  This procedure is done efficiently by using the data structures commonly available in domain decomposition codes and immersed boundary methods. It requires to perform a loop on the edges objects (which are usually identified anyway in a domain decomposition code), and determine if those edges are intersected by the interface or not. Intersection points are usually already pre-computed in immersed boundary codes. As a result, no further significant overhead is introduced. Moreover, this operation has to be performed only in the subdomains cut by the boundary, which is typically a small set. The numerical examples (see Sect.\ref{sec:examples}) show that by splitting edge objects in this way, optimal weak scaling is recovered if Neumann boundary conditions are imposed on the unfitted parts of the domain boundary $\partial\Omega$.

\subsection{New coarse objects: variant 2}
\label{sec:new-objects-v2} 

However, the case of weak imposition of Dirichlet conditions has been observed to be more challenging and requires further elaboration. In that case, we propose a second strategy to split the edges. {The method is motivated by the fact that the  domain decomposition analysis of non-spectral methods relies on trace theorems, and it requires constant weighting coefficients within each coarse object in order to derive condition number bounds (see {Eq. \eqref{eq:weicon}} and, e.g., \cite[p. 141]{toselli_domain_2004}, \cite{mandel_convergence_2003}, or \cite[p. 208]{brenner_mathematical_2010}).}\footnote{{We note that the numerical analysis of more involved non-diagonal weighting operators (e.g., deluxe scaling, based on spectral methods) does not rely on trace theorems.}}  We note that the coefficient  $\beta$ required to ensure coercivity in Nitsche's method might be very different between neighbor elements (even different orders of magnitude) in function of the location of the interface. This induces strong variations in the coefficients of the stiffness matrix, and therefore highly oscillatory weighting coefficients, when they are computed as in \eqref{eq:alge-weighting}. As a result, the weighting coefficient might be non-constant within the BDDC objects cut by the interface (see Fig.~\ref{fig:new-corners-2:a}), which violates one of the basic assumptions of the domain decomposition analysis. Motivated by this fact, we try to build objects that have constant  weighting coefficients. To this end, we follow two steps.  First, we build the standard BDDC objects (Sect.~\ref{sec:std-objects}). Then, we split all coarse objects classified as \emph{edges} that have non-constant weighting coefficients. Precisely, we aggregate neighbor nodes on the edge that have the same weighting coefficient {(note that many nodes on a coarse edge have the same coefficient, up to machine precision, for a structured meshes and linear elements)}, and then we set the remaining un-aggregated nodes as coarse corners (see Fig.~\ref{fig:new-corners-2:b}). Since edges are one dimensional curves, this aggregation can be easily performed. Moreover, this process is performed only for the small set of subdomains that are cut by the boundary. %We keep the faces untouched, because face objects are not strictly needed in the BDDC analysis and it would be computationally expensive to treat them. 
By modifying the coarse objects in this way, an optimal weak scaling is also recovered for weak Dirichlet boundary conditions, as showed in the following 3D numerical examples.

\begin{figure}[ht!]

\centering

{\color{blue}$\bullet$} Nodes on coarse edges \hspace{0.5cm} {\color{red} $\boldsymbol{\times}$} Nodes on coarse corners

\begin{subfigure}[b]{0.33\textwidth}
\includegraphics[width=\textwidth]{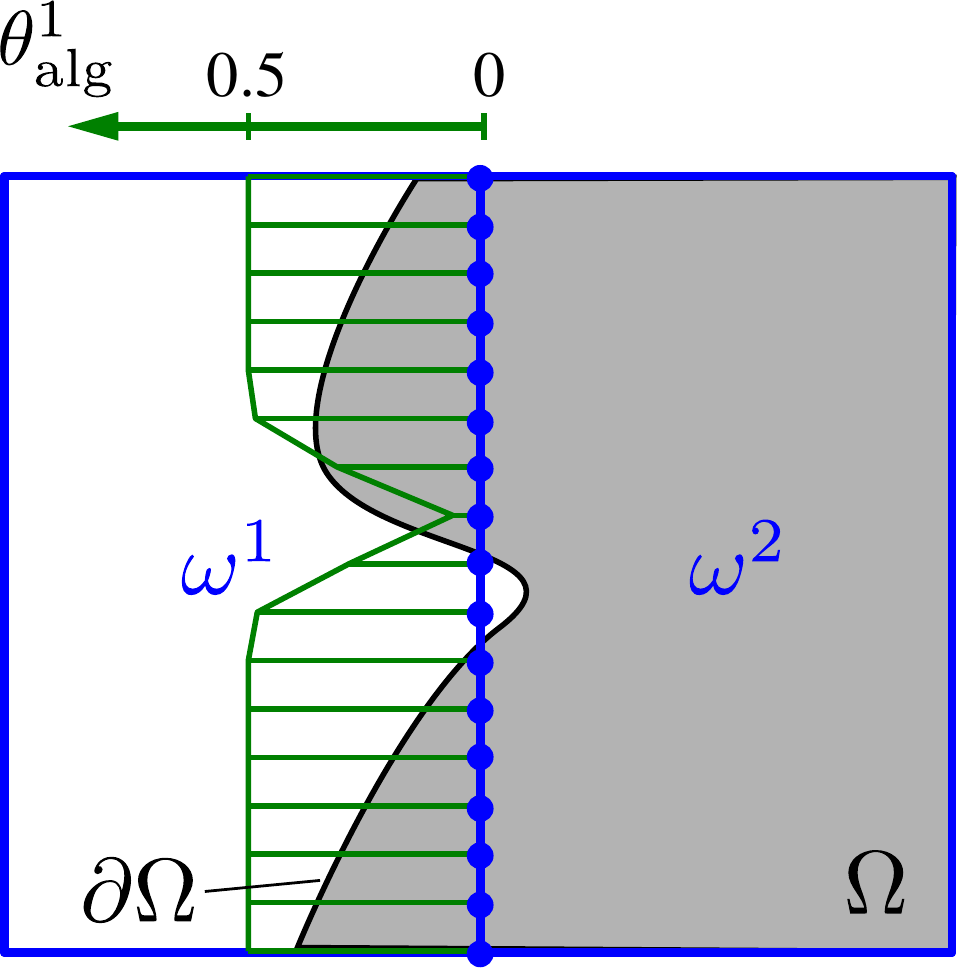}
\caption{Before}
\label{fig:new-corners-2:a}
\end{subfigure}
\hspace{0.5cm}
\begin{subfigure}[b]{0.33\textwidth}
\includegraphics[width=\textwidth]{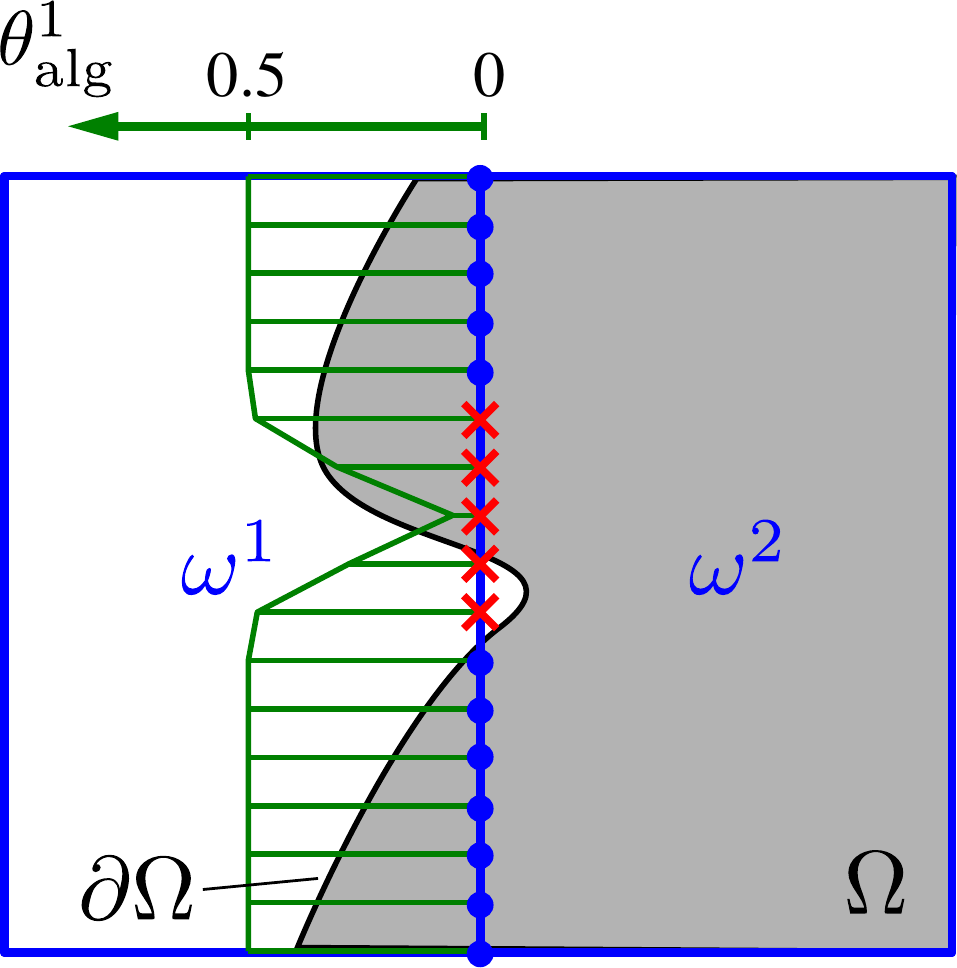}
\caption{After}
\label{fig:new-corners-2:b}
\end{subfigure}
\caption{Splitting a cut edge into several uncut edges (second variant).}
\label{fig:new-corners-2}
\end{figure}

\section{Numerical experiments}
\label{sec:examples}
% FRANCESC

%All the examples by Francesc.

%Goal of the numerical experiments is to show that:
%\begin{itemize}
%\item Proposed BDDC methods are weakly scalable for cut elements.
%\item The proposed corner-detection mechanisms do not introduce too many dofs in the coarse space.
%\item Scaling properties nearly the same as standard bddc for usual body fitted meshes.
%\item Methods applicable to different complex geometries.
%\item Methods robust w.r.t. position of the cut elements.
%\item Effect on processor load balancing.
%\end{itemize}

%The numerical examples below study the performance of the proposed BDDC preconditioners when solving the system of linear equations arising from the discretization of the Poisson equation using the embedded FE methods presented in Sect. \TBD{} ...

\subsection{Setup}

%Before studying the performance of the proposed methods, we detail here the setup used in the numerical experiments. 

In the numerical examples below, we consider the Poisson problem with the discretization being used in Sect. \ref{sec:modpr} defined on five different complex 3D domains shown in Fig.~\ref{fig:geometries}. The chosen geometries are 1) a sphere, 2) a body that reminds to a popcorn flake, 3) a hollow block, 4) a 3-by-3-by-3 array of such blocks, and 5) a spiral.  These geometries are used often in the literature to study the performance of unfitted FE methods (see, e.g., \cite{burman_cutfem:_2015}).  This variety of shapes is considered here to evaluate the versatility of the proposed solvers and their ability  to deal with complex and diverse 3D domains. {A detailed numerical experimentation} is basic, since the solvers proposed are based on heuristic motivations, and numerical experimentation is the only way to prove their robustness. %The precise definition of each piece is given in Annex \TBD{}  {\color{red} Annex is TODO}. 
\begin{figure}[ht!]
\centering
\begin{subfigure}[b]{0.19\textwidth}
\includegraphics[width=\textwidth]{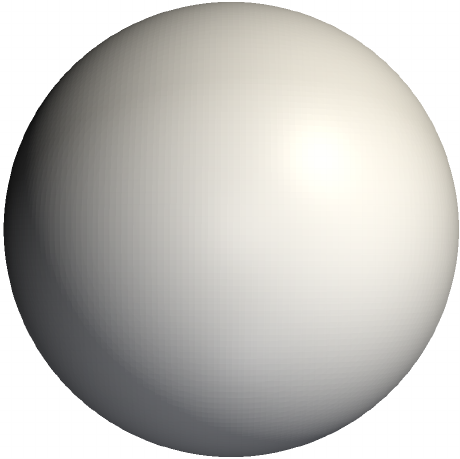}
\caption{Sphere.}
\end{subfigure}
\begin{subfigure}[b]{0.19\textwidth}
\includegraphics[width=\textwidth]{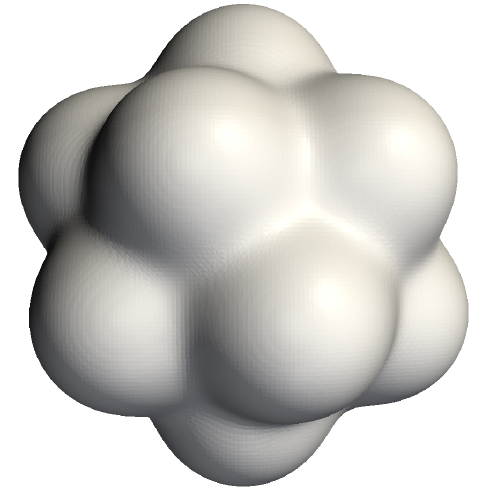}
\caption{Popcorn flake.}
\label{fig:geometries:popcorn}
\end{subfigure}
\begin{subfigure}[b]{0.19\textwidth}
\includegraphics[width=\textwidth]{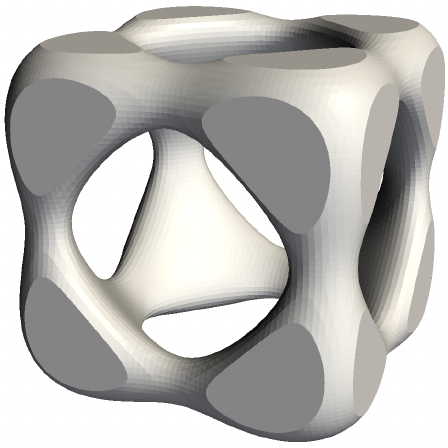}
\caption{Block.}
\end{subfigure}
\begin{subfigure}[b]{0.20\textwidth}
\includegraphics[width=\textwidth]{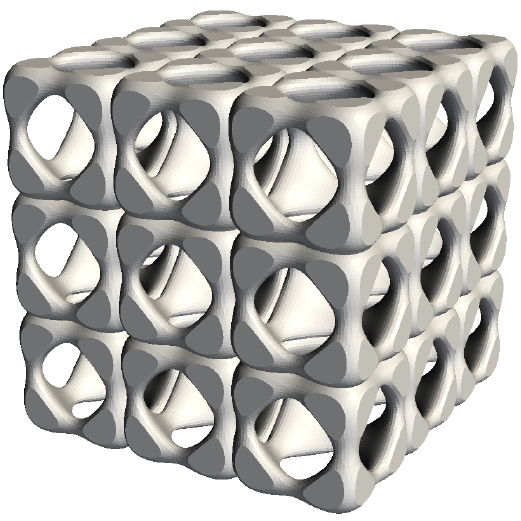}
\caption{Array of blocks.}
\end{subfigure}
\begin{subfigure}[b]{0.15\textwidth}
\includegraphics[width=\textwidth]{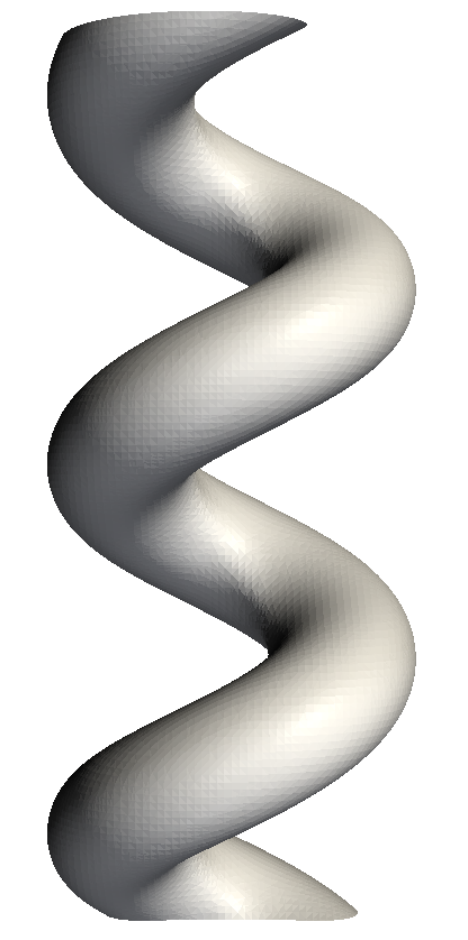}
\caption{Spiral.}
\end{subfigure}
\caption{Geometries used in the numerical examples.}
\label{fig:geometries}
\end{figure} 

%For each piece, the associated computational meshes and partitions into subdomains are systematically generated as indicated in section XX.  

For each geometry, weak scalability tests are executed in order to evaluate the performance of the solvers. To this end, a family of computational meshes are generated for each piece spanning from coarse meshes up to  fine meshes with several millions of elements (see Table~\ref{table:meshes} for details). Note that the subdomain partitions are refined at the same rate than the meshes, leading to a constant ratio $H/h$, which should lead to an optimal weak scaling of the BDDC preconditioner. That is, the number of iterations of the underlying preconditioned conjugate gradient solver is expected to be (asymptotically) independent of the problem size. The main purpose of the numerical example below is to check that the proposed BDDC preconditioners for unfitted meshes have this optimal performance. In the remainder of this sub-section, we further detail the setup of the numerical examples.
\begin{table}[ht!]
\caption{Information of the computational meshes used in the examples. }
\label{table:meshes}
\begin{small}
\begin{tabular}{lcrrrrrrrr}
\toprule
 & Id.  & $n^\mathrm{sd}$  & $n^\mathrm{el}$ & $n^\mathrm{sdel}$  & $n^\mathrm{dof}$ &  $n^\mathrm{sddof}$   &  $h$ & $H$ &  $H/h$\\
\midrule
       Sphere & 1 &       8 &    1064 &     512 &    1461 &     729 &  0.1250 &  1.0000 &     8.0 \\
              & 2 &      32 &    7160 &     512 &    8577 &     729 &  0.0625 &  0.5000 &     8.0 \\
              & 3 &     160 &   51968 &     512 &   57123 &     729 &  0.0312 &  0.2500 &     8.0 \\
              & 4 &    1064 &  396040 &     512 &  415775 &     729 &  0.0156 &  0.1250 &     8.0 \\
              & 5 &    7160 & 3090104 &     512 & 3167421 &     729 &  0.0078 &  0.0625 &     8.0 \\
\midrule
Popcorn flake & 1 &       8 &    1920 &     512 &    2559 &     729 &  0.1125 &  0.9000 &     8.0 \\
              & 2 &      60 &   12936 &     512 &   15181 &     729 &  0.0563 &  0.4500 &     8.0 \\
              & 3 &     324 &   95256 &     512 &  103701 &     729 &  0.0281 &  0.2250 &     8.0 \\
              & 4 &    1924 &  729768 &     512 &  762497 &     729 &  0.0141 &  0.1125 &     8.0 \\
              & 5 &   12936 & 5708900 &     512 & 5837623 &     729 &  0.0070 &  0.0563 &     8.0 \\
\midrule
        Block & 1 &       8 &    2120 &     512 &    3280 &     729 &  0.2312 &  1.8500 &     8.0 \\
              & 2 &      56 &   14336 &     512 &   18772 &     729 &  0.1156 &  0.9250 &     8.0 \\
              & 3 &     304 &  103152 &     512 &  120164 &     729 &  0.0578 &  0.4625 &     8.0 \\
              & 4 &    2120 &  776584 &     512 &  842652 &     729 &  0.0289 &  0.2312 &     8.0 \\
              & 5 &   14336 & 6011832 &     512 & 6272228 &     729 &  0.0145 &  0.1156 &     8.0 \\
\midrule
        Array & 1 &     216 &   53904 &     512 &   72796 &     729 &  0.2271 &  1.8167 &     8.0 \\
              & 2 &    1512 &  376696 &     512 &  451516 &     729 &  0.1135 &  0.9083 &     8.0 \\
              & 3 &    8208 & 2747880 &     512 & 3064716 &     729 &  0.0568 &  0.4542 &     8.0 \\
\midrule
       Spiral & 1 &       2 &     435 &     512 &     838 &     729 &  0.1500 &  1.2000 &     8.0 \\
              & 2 &      16 &    2277 &     512 &    3501 &     729 &  0.0750 &  0.6000 &     8.0 \\
              & 3 &      96 &   14309 &     512 &   18511 &     729 &  0.0375 &  0.3000 &     8.0 \\
              & 4 &     441 &   99991 &     512 &  115378 &     729 &  0.0188 &  0.1500 &     8.0 \\
              & 5 &    2277 &  744170 &     512 &  802975 &     729 &  0.0094 &  0.0750 &     8.0 \\
              & 6 &   14309 & 5733813 &     512 & 5963098 &     729 &  0.0047 &  0.0375 &     8.0 \\
\bottomrule
\end{tabular}

\begin{tabular}{ll}
Legend\\
$n^\mathrm{sd}$: number of subdomains & $n^\mathrm{el}$: total number of elements \\
$n^\mathrm{sdel}$: max. number of elements in a subdomain & $n^\mathrm{dof}$: total number of DOFs\\
 $n^\mathrm{sddof}$: max. number of DOFs in a subdomain & $h$: element size,  $H$: subdomain size.
\end{tabular} 
\end{small}
\end{table}

We define the source term and boundary conditions of the Poisson equation such that the PDE has an exact solution, namely
\begin{equation}
u(x,y,z) = \sin\left(5\pi\left(x^2 + y^2 + z^2\right)^{1/2}\right), \quad (x,y,z)\in\Omega\subset\mathbb{R}^3.
\end{equation}
We use the exact solution to assess the quality of the numerical approximations provided by the solvers. In order to define the boundary conditions, we partition the boundary $\partial\Omega$ into two non-overlapping sets $\Gamma_\mathrm{bb}$ and $\Gamma_\mathrm{cut}$ (see  Fig.~\ref{fig:boundaries}). The set $\Gamma_\mathrm{bb}$ is the intersection of $\partial\Omega$ with the bounding box used to define the computational mesh, and $\Gamma_\mathrm{cut}$ is its complement, $\Gamma_\mathrm{cut}\doteq \partial\Omega - \Gamma_\mathrm{bb}$. After discretizing the bounding box with a Cartesian mesh, the set $\Gamma_\mathrm{bb}$ results to be aligned with the mesh faces. On the contrary, $\Gamma_\mathrm{cut}$ is a non-conforming boundary living inside the cut elements. We define strong Dirichlet conditions on the boundary $\Gamma_\mathrm{bb}$. Since $\Gamma_\mathrm{bb}$ is aligned with the mesh, strong Dirichlet boundary conditions are imposed by constraining the corresponding mesh nodes. On the other hand, we impose two different types of boundary conditions on the unfitted surface $\Gamma_\mathrm{cut}$, either weak Dirichlet boundary conditions using Nitsche's method or Neumann boundary conditions. Note that for the popcorn flake and the spherical domain, the set $\Gamma_\mathrm{bb}$ is empty since all the boundary $\partial\Omega$ in non-conforming  (see Fig.~\ref{fig:boundaries}). In these cases, we strongly impose the value of the solution at an interior node of $\Omega$ in order to have a well posed problem when considering Neumann conditions on the entire boundary.
\begin{figure}[ht!]
\centering
{\small
{\color[RGB]{0,170,255} \rule{1cm}{8pt}} $\Gamma_\mathrm{bb}$ {\color[rgb]{0.7,0.7,0.7} \rule{1cm}{8pt}} $\Gamma_\mathrm{cut}$\\[-0.4cm]  
}

\begin{subfigure}[b]{0.19\textwidth}
\includegraphics[width=\textwidth]{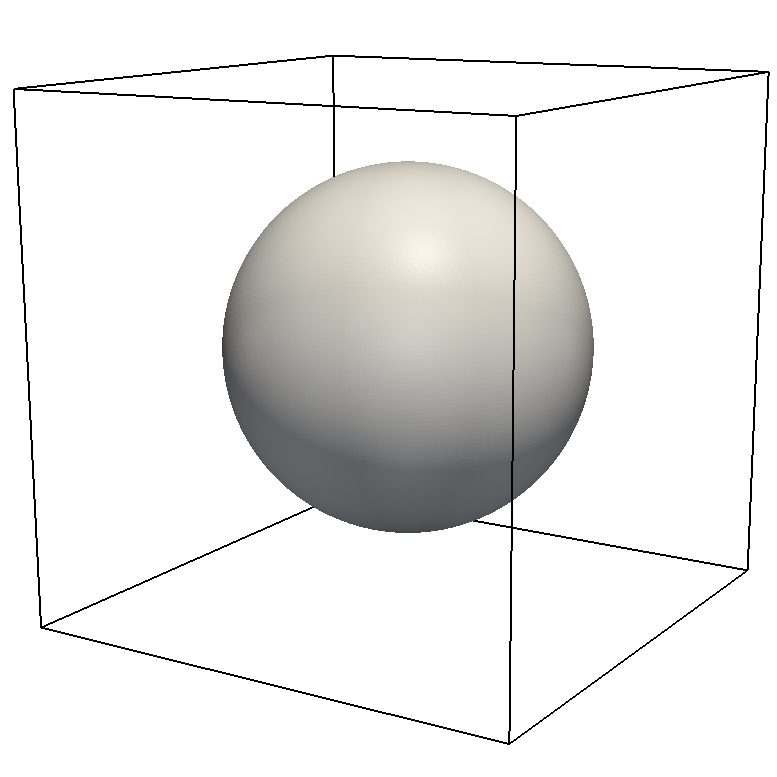}
\caption{Sphere.}
\end{subfigure}
\begin{subfigure}[b]{0.19\textwidth}
\includegraphics[width=\textwidth]{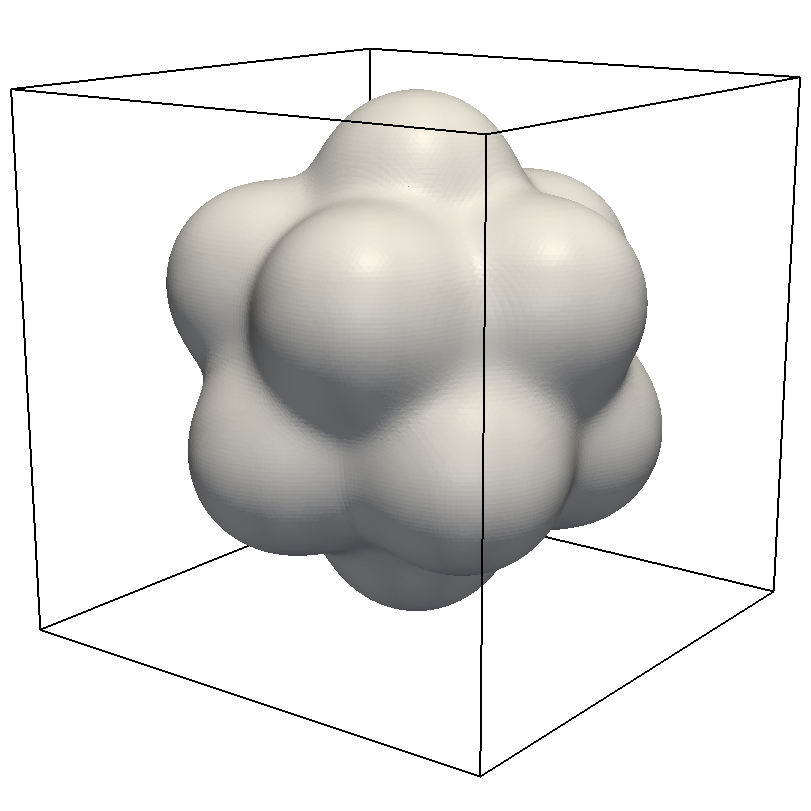}
\caption{Popcorn flake.}
\end{subfigure}
\begin{subfigure}[b]{0.19\textwidth}
\includegraphics[width=\textwidth]{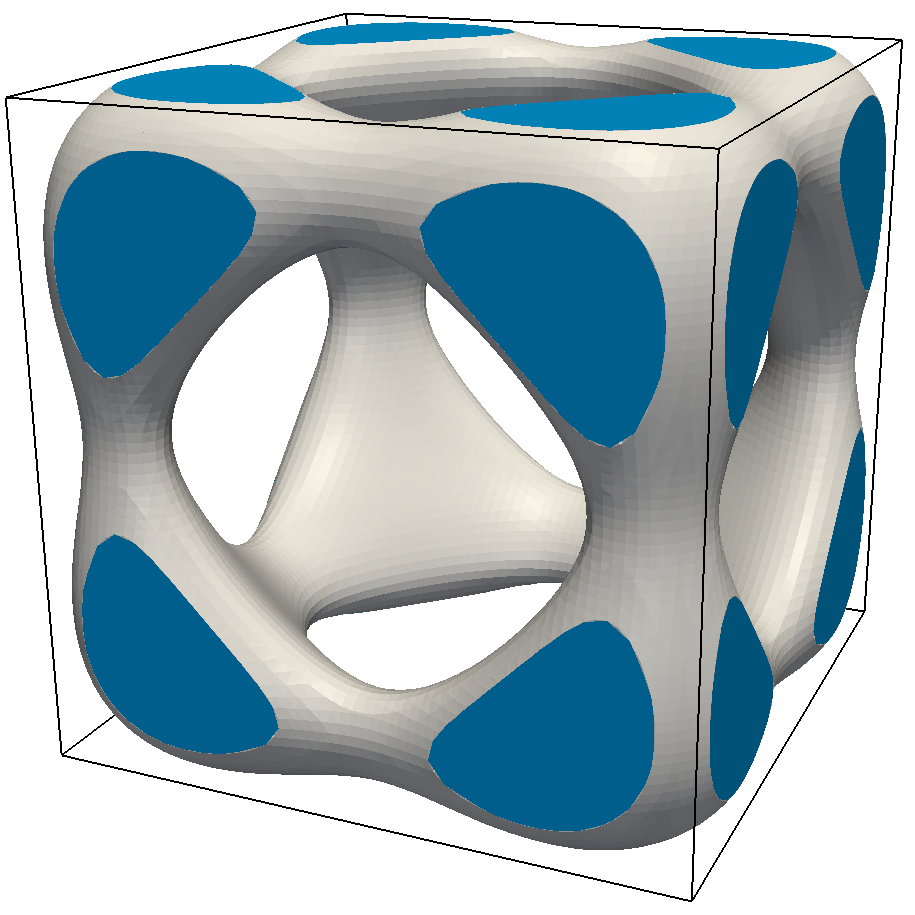}
\caption{Block.}
\end{subfigure}
\begin{subfigure}[b]{0.20\textwidth}
\includegraphics[width=\textwidth]{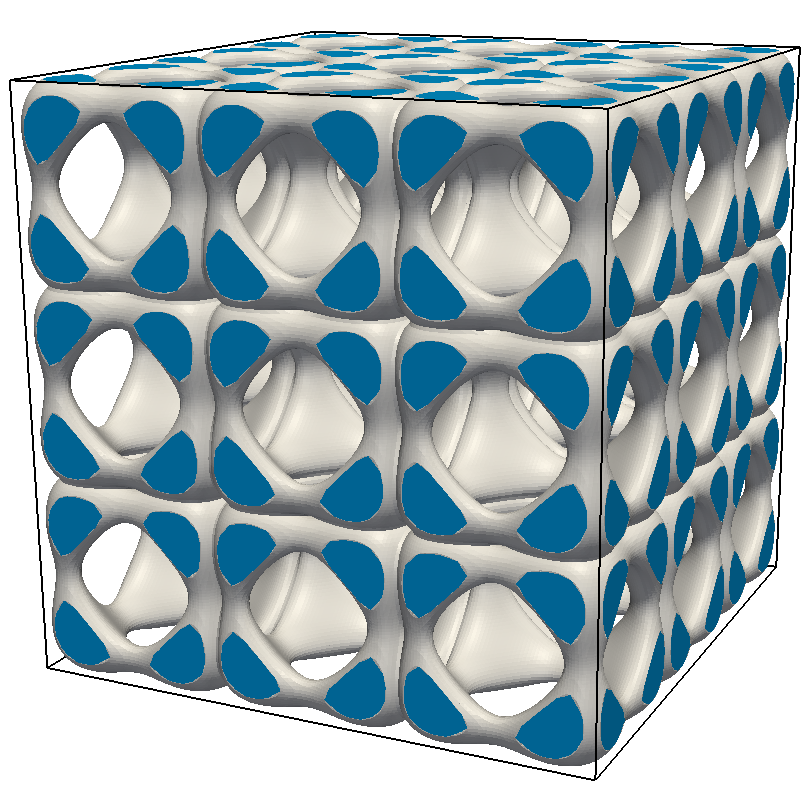}
\caption{Array of blocks.}
\end{subfigure}
\begin{subfigure}[b]{0.16\textwidth}
\includegraphics[width=\textwidth]{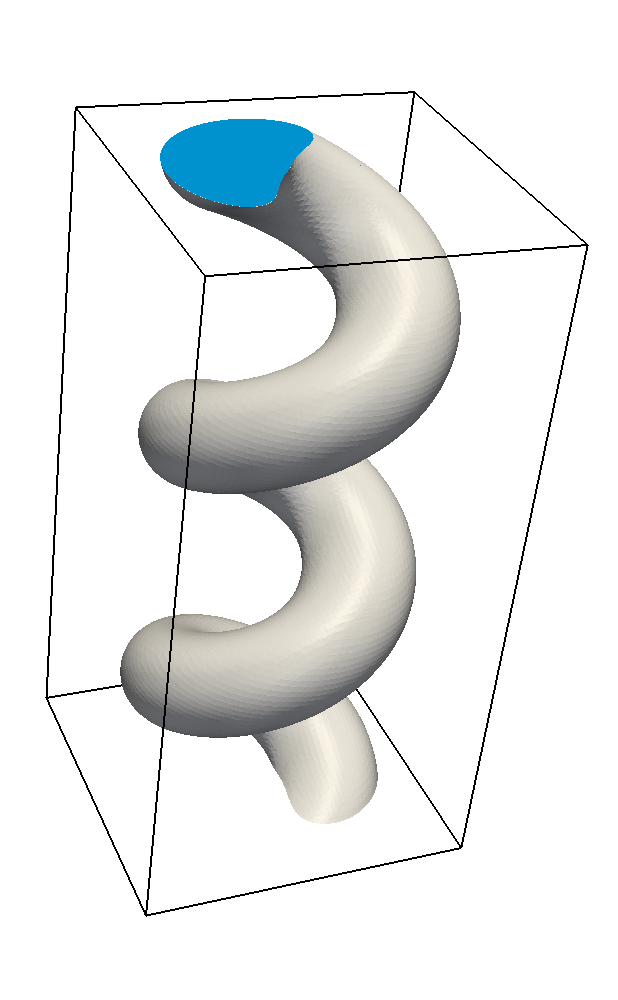}
\vspace{-0.7cm}
\caption{Spiral.}
\end{subfigure}
\caption{Illustration of the sets $\Gamma_\mathrm{bg}$ and $\Gamma_\mathrm{cut}$.}
\label{fig:boundaries}
\end{figure} 

After the discretization process, the resulting systems of linear equations $Ax=b$ are solved with a BDDC-preconditioned conjugate gradient (PCG) solver. %The preconditioners are the proposed BDDC methods in Sect.~\ref{sec:bddc-unfitted}. 
Convergence of the PCG solver is declared when the 2-norm of the algebraic residual $r\doteq b-Ax$ is smaller that the 2-norm of the right-hand-side $b$ times a tolerance, namely $|r|_2 < \mathrm{tol}\cdot|b|_2$. In all cases, the selected tolerance is $\mathrm{tol}=10^{-9}$. With this value, the error committed in solving the linear system is negligible in front of the discretization error associated with the FE mesh. This is demonstrated by the fact that the optimal convergence rate of FEM is obtained by using this tolerance for all the geometries (see Fig.~\ref{fig:errorconv}).  %Thus, the chosen tolerance is strict enough and the number of iterations of the PCG solver is a meaningful indicator to evaluate the performance of the preconditioners.
\begin{figure}[ht!]
\centering
\begin{subfigure}{0.33\textwidth}
\centering
\includegraphics[width=\textwidth]{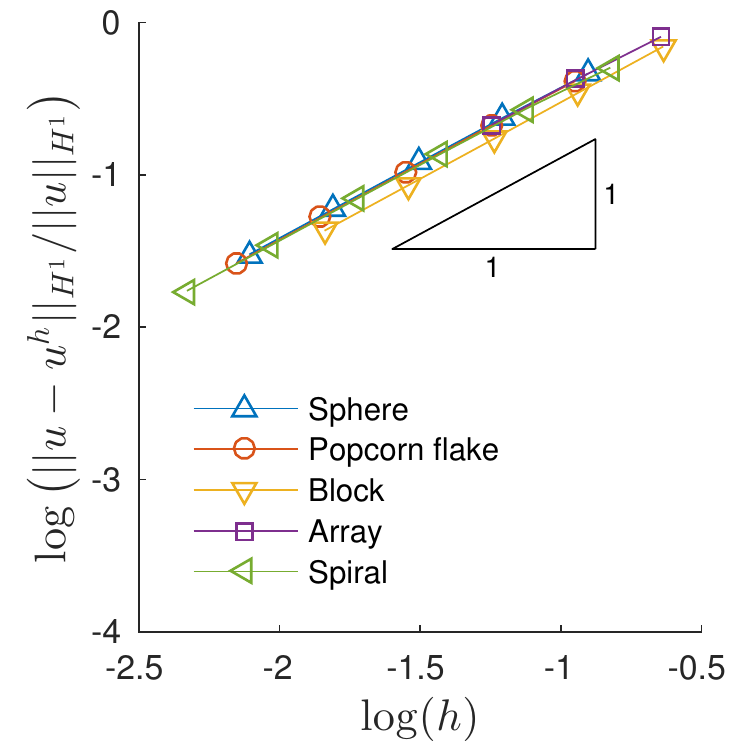}
\end{subfigure}
\begin{subfigure}{0.33\textwidth}
\centering
\includegraphics[width=\textwidth]{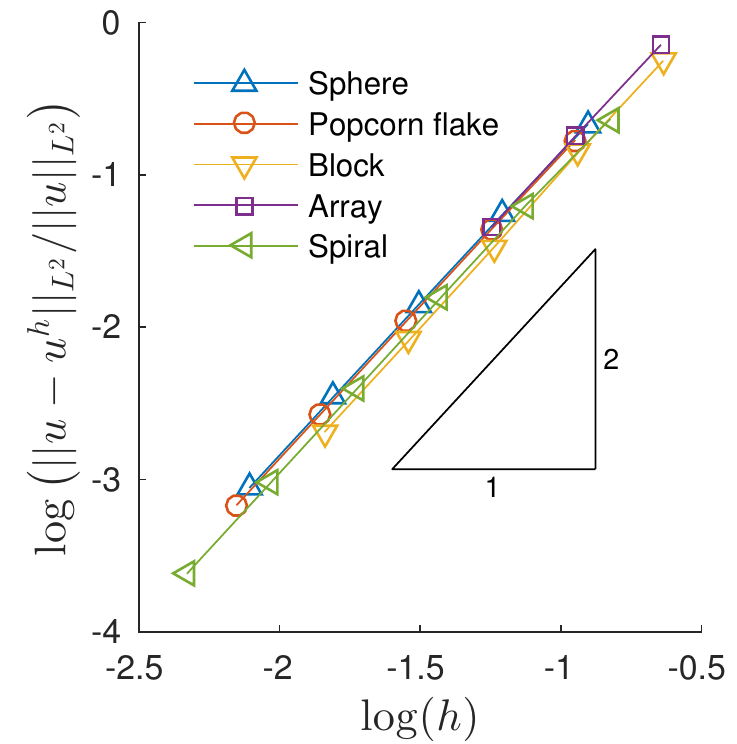}
\end{subfigure}
\caption{Convergence of the discretization error in $H^1$ and $L^2$ norms.}
\label{fig:errorconv}
\end{figure}

Six different versions of the BDDC preconditioner are studied in the examples, see Table~\ref{table:methods}. They differ in the definition of the coarse objects (either the Standard definition of section~\ref{sec:std-objects} or the extended definitions for unfitted meshes in Sects.~\ref{sec:new-objects-v1} and~\ref{sec:new-objects-v2}) and  the subset of objects used to define the coarse DOFs (either corners and edges, or corners edges and faces). In all the cases, the algebraic weighting operator of Sect.~\ref{sec:new-weighting} is considered. {We do not include the topological weighting operator (detailed in Sect.~\ref{sec:topo-weighting}) since it leads to a solver that is not robust with respect to the relative position between the geometry and the mesh  (see, e.g., the breakdown example in Sect. \ref{sec:breackdown-ex}). We have experimentally observed that the topological weighting operator leads to unacceptably large iteration counts for all the 3D geometries studied here. In particular, more than 1000 iterations are needed to solve the ``popcorn flake'' with a medium-size mesh (mesh \# 3 described in Table \ref{table:meshes}) for the case of Neumann boundary conditions on $\Gamma_\mathrm{cut}$.} {Thus, the breakdown of the method does not only appear in \emph{manufactured} examples as the one in \ref{sec:breackdown-ex} but in all practical examples analyzed in this section.}

Finally, in order to compare the performance of the proposed BDDC methods for unfitted grids, with respect to the optimal behavior of BDDC for conventional conforming meshes, we construct auxiliary  meshes by replacing the cut elements by standard full elements (see Fig.~\ref{fig:popcorn_cut_vs_full}). Our goal is to recover a similar performance for cut elements, that the one obtained with full elements.

\begin{table}[ht!]
\caption{Description of the BDDC preconditioners used in the numerical examples.}
\label{table:methods}
\begin{small}
\begin{tabular}{llll}
\toprule
     Identifier                   &  Object definition & Selected objects     & Weighting operator\\
 \midrule
 Standard BDDC(cef)      & As in Sect.~\ref{sec:std-objects} & corners, edges, faces & As in Sect.~\ref{sec:new-weighting}\\
 Standard BDDC(ce)              & As in Sect.~\ref{sec:std-objects} & corners, edges & As in Sect.~\ref{sec:new-weighting}\\
 1st variant BDDC(cef)   & As in Sect.~\ref{sec:new-objects-v1} & corners, edges, faces & As in Sect.~\ref{sec:new-weighting}\\
1st variant BDDC(ce)           & As in Sect.~\ref{sec:new-objects-v1} & corners, edges  & As in Sect.~\ref{sec:new-weighting}\\
2nd variant BDDC(cef)   & As in Sect.~\ref{sec:new-objects-v2} & corners, edges, faces  & As in Sect.~\ref{sec:new-weighting}\\
2nd variant BDDC(ce)          & As in Sect.~\ref{sec:new-objects-v2} & corners, edges   & As in Sect.~\ref{sec:new-weighting}\\
 \bottomrule
\end{tabular}
\end{small}
\end{table}

\begin{figure}[ht!]
\centering
\begin{subfigure}{0.35\textwidth}
\centering
\includegraphics[width=\textwidth]{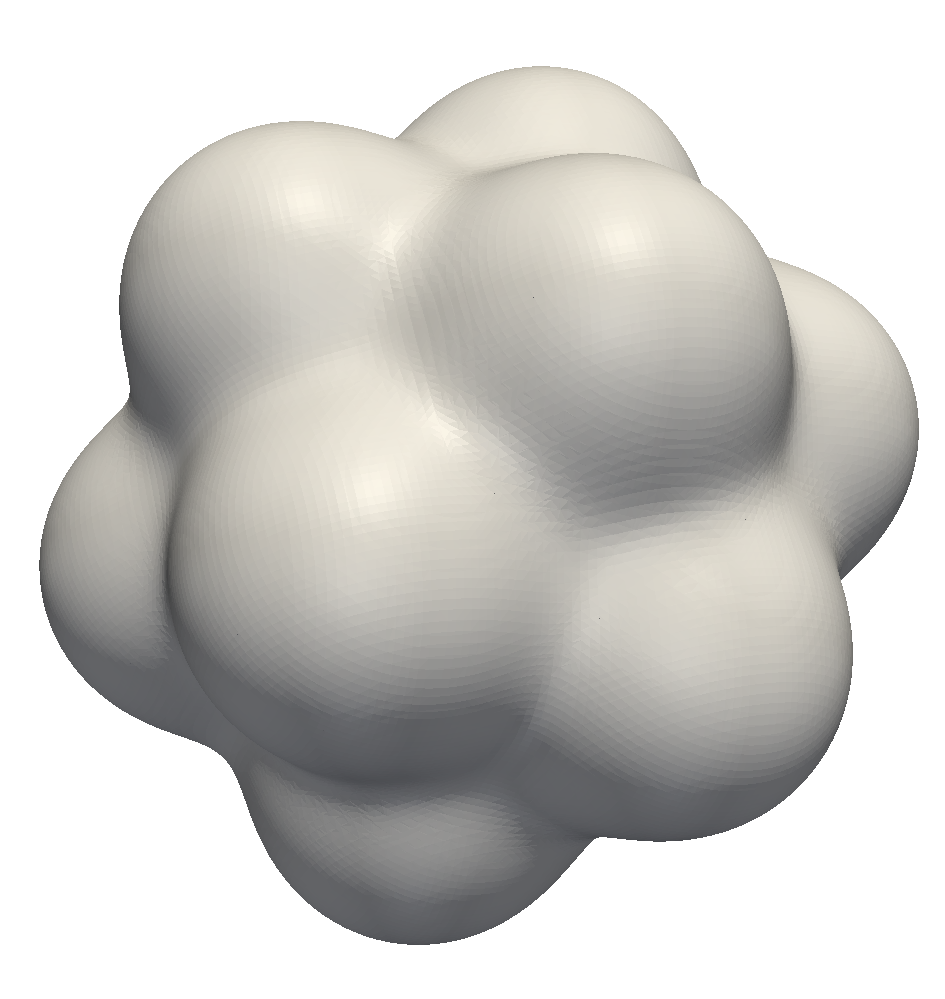}
\caption{Cut elements.}
\end{subfigure}
\begin{subfigure}{0.35\textwidth}
\centering
\includegraphics[width=\textwidth]{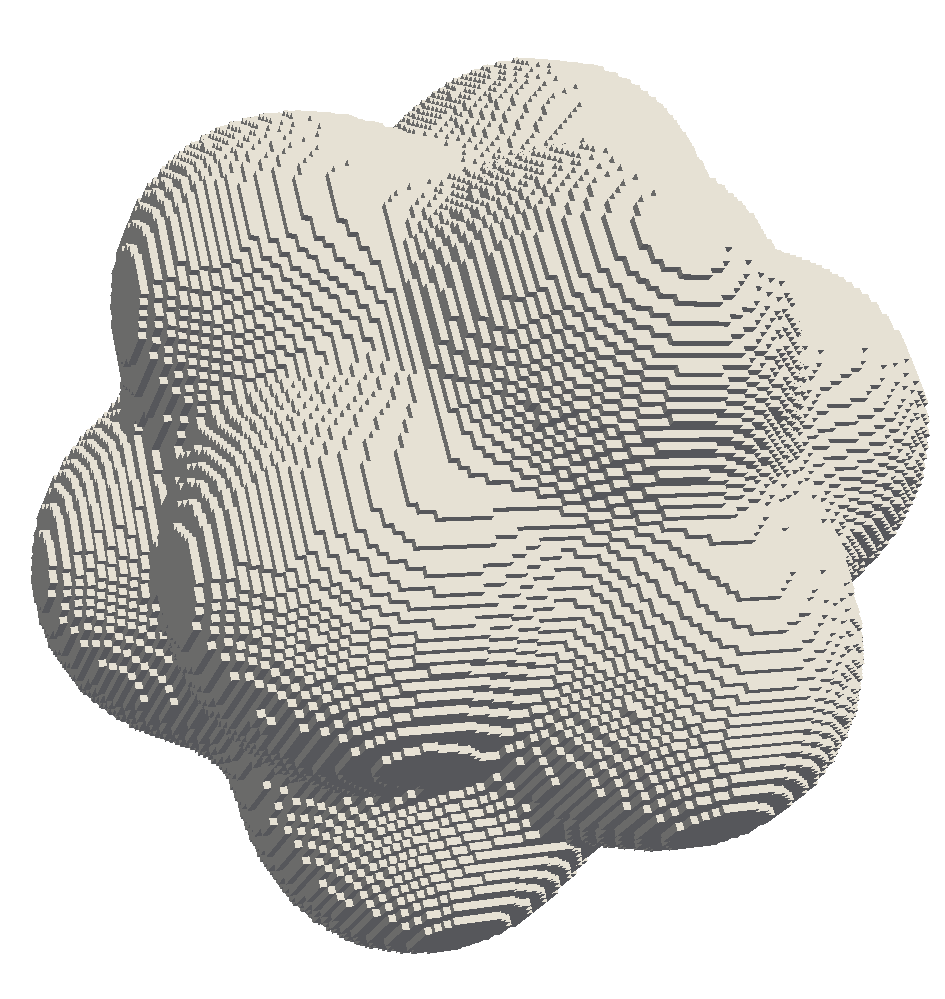}
\caption{Full elements.}
\end{subfigure}
\caption{The ``popcorn flake'' geometry reconstructed with cut elements and with full elements.}
\label{fig:popcorn_cut_vs_full}
\end{figure}

\newcommand{\examplename}{``Popcorn flake'' example}
\subsection{``Popcorn flake'' example.} For the sake of clarity, we first present the numerical results for a single geometry, namely the ``popcorn flake'' geometry in Fig.~\ref{fig:geometries:popcorn}. The results for the other shapes are discussed in Sect.~\ref{sec:multi-body} below.  Fig.~\ref{fig:popcorn_wscal_iters} shows the result of the weak scaling test for this geometry, where  Fig.~\ref{fig:popcorn_wscal_iters_nbc} is for the case of Neumann boundary conditions on $\Gamma_\mathrm{cut}$, and Fig.~\ref{fig:popcorn_wscal_iters_dbc} is for weak Dirichlet boundary conditions.  It is observed that the standard BDDC method leads to a perfect weak scaling for the case of full elements (as expected) for the two types of boundary conditions. %, i.e. the number of iterations tend to be constant as the mesh is refined. 
The performance of the standard BDDC method when dealing with cut elements (i.e., red lines in Fig.~\ref{fig:popcorn_wscal_iters}) is decent, in the sense that the number of iterations increases only mildly with the problem size, but it is clear that the absolute number of iterations needed to solve the problems is more than twice the iterations required for full elements (which is the target value). In contrast, both the 1st and 2nd variants of BDDC achieve a perfect weak scaling with iteration counts close to the ones associated with full elements for the case of Neumann boundary contitions (see blue and yellow lines in Fig.~\ref{fig:popcorn_wscal_iters_nbc}). 

\begin{figure}[ht!]
\centering
\begin{subfigure}{\textwidth}
\centering
\includegraphics[width=0.9\textwidth]{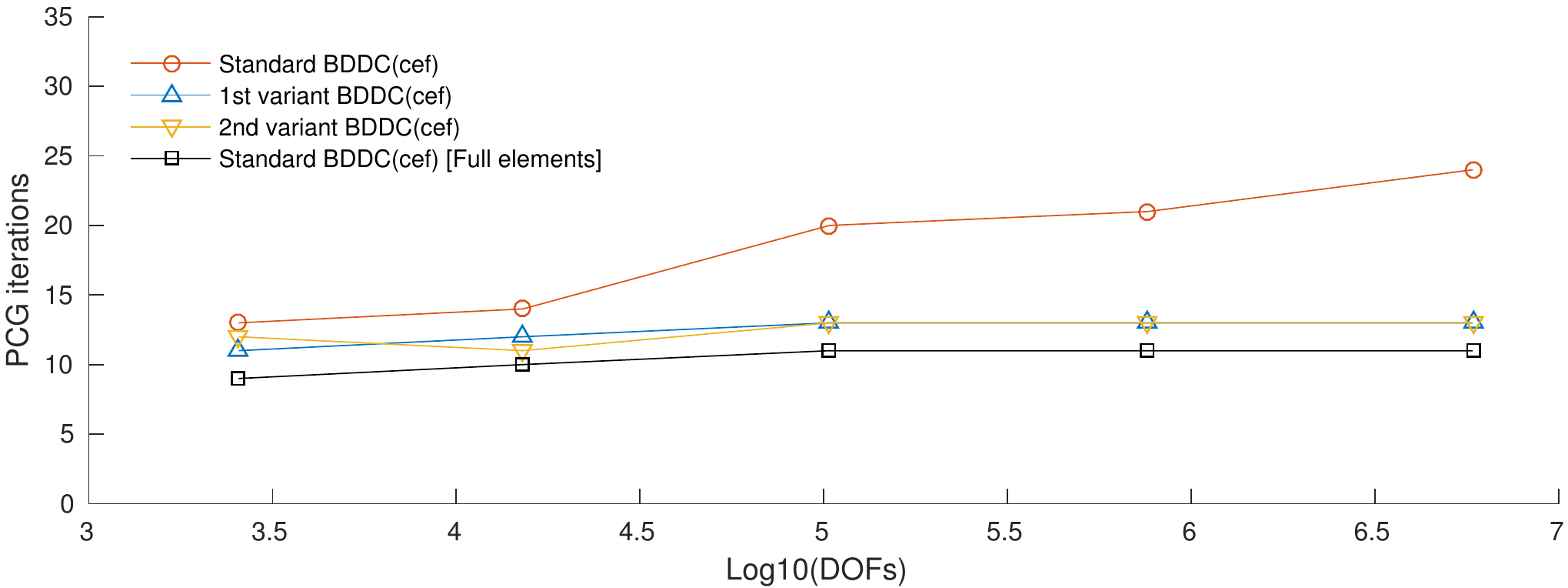}
\caption{Results for Neumann boundary conditions.}
\label{fig:popcorn_wscal_iters_nbc}
\end{subfigure}
\begin{subfigure}{\textwidth}
\centering
\includegraphics[width=0.9\textwidth]{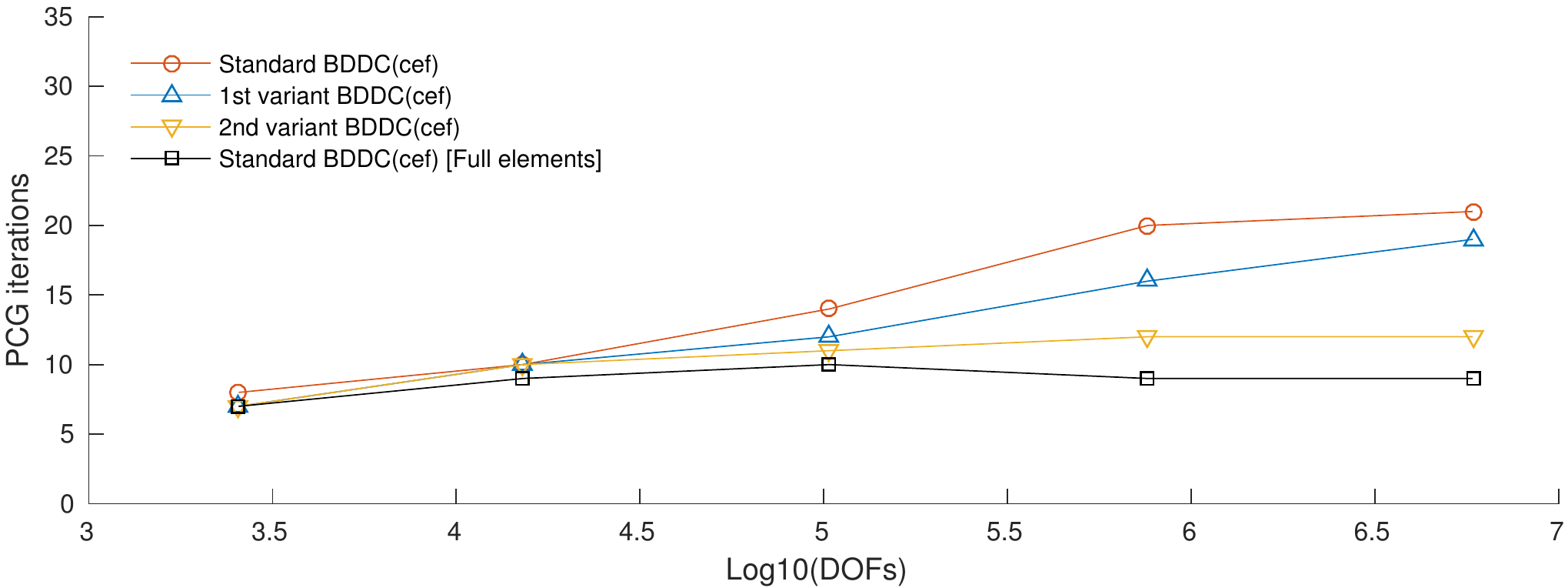}
\caption{Results for Dirichlet boundary conditions.}
\label{fig:popcorn_wscal_iters_dbc}
\end{subfigure}
\caption{\examplename: Weak scalability test.}
\label{fig:popcorn_wscal_iters}
\end{figure} 
However, for Dirichlet boundary conditions, only the 2nd variant provides a perfect weak scaling close to the one observed for full elements (see yellow line in Fig.~\ref{fig:popcorn_wscal_iters_dbc}).  %It is not surprising that the case of Dirichlet boundary conditions is more challenging. The coefficient $\beta$ introduced by Nitsche's method might be very different between neighbor elements (even different orders of magnitude). These strong changes in the coefficient makes the problem more difficult for BDDC, as it is known that conventional BDDC methods suffer in highly heterogeneous problems, [Ref PB-BDDC]. 
In conclusion, the 1st variant of the proposed method is enough to achieve a perfect weak scaling for Neumann boundary conditions, whereas  the 2nd variant is required for Dirichlet conditions imposed with Nitsche's method.  Similar results are obtained if the coarse space is defined with corners, edges, and faces (cef), or only with corners and edges (ce) (see  Fig.~\ref{fig:popcorn_wscal_iters_cef_vs_ce}). The main difference is that the iterations associated with corners and edges are slightly larger than for corners, edges, and faces, which is the expected behavior in BDDC methods. The iteration count is smaller in the later case since the associated coarse space is larger.
\begin{figure}[ht!]
\centering
\includegraphics[width=0.9\textwidth]{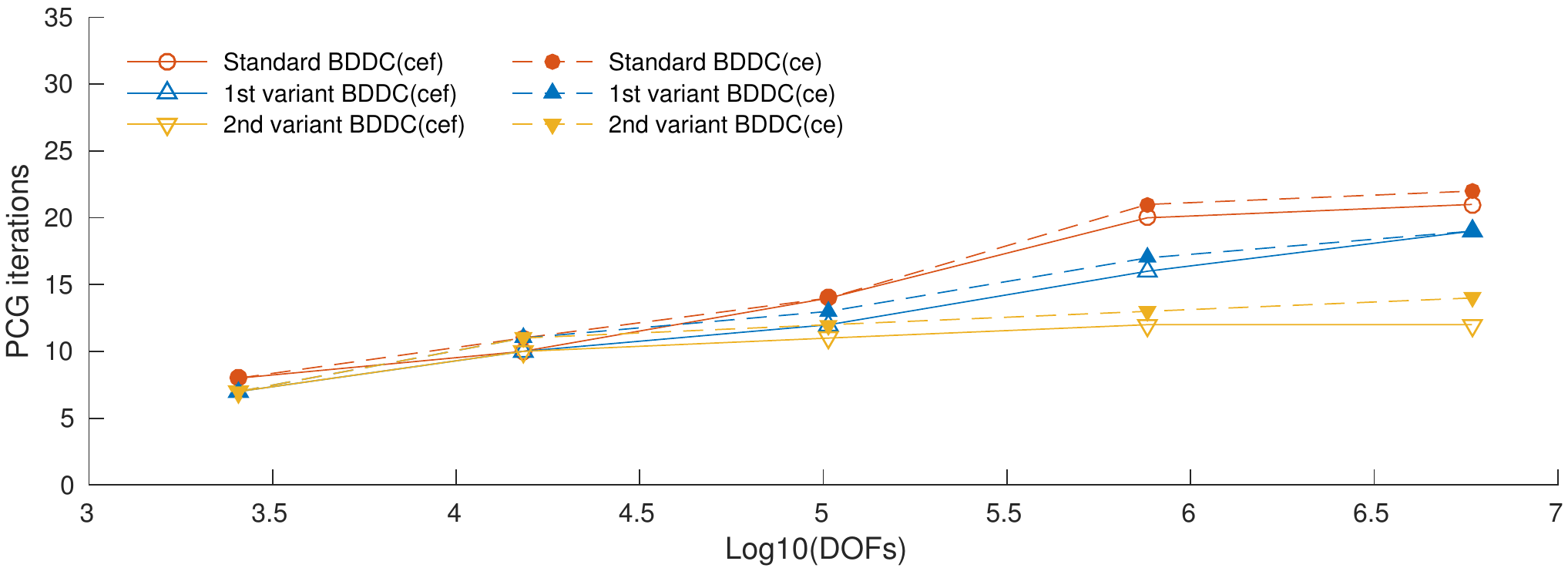}
\caption{\examplename: Weak scalability test. Comparison between BDDC(cef) and BDDC(ce).}
\label{fig:popcorn_wscal_iters_cef_vs_ce}
\end{figure}

Fig.~\ref{fig:popcorn_wscal_cdof_cef_vs_ce} shows the size of the coarse space of the new BDDC methods, namely variants 1 and 2, with respect to the size of the standard coarse space. %A value of 1.5 in Fig.~\ref{fig:popcorn_wscal_cdof_cef_vs_ce} means that the coarse space is increased 1.5 times due to splitting the edges, and a value of 1.0 would mean that the coarse space is not increased at all. 
Fortunately,   the size of the coarse spaces associated with variants 1 and 2 tends to be equal to the size of the standard coarse space in BDDC as the mesh is refined (see Fig.~\ref{fig:popcorn_wscal_cdof_cef_vs_ce}). This behavior is explained by the fact that extra corners are only added on edges that are close to cut elements. Since the number of BDDC objects in the interior of $\Omega$ scales much faster that the  objects cut by the boundary $\partial\Omega$,   the number of extra corners added near the boundary become nearly negligible as the partition is refined. As a result, the finer is the mesh the more competitive are the proposed BDDC preconditioners for unfitted meshes. Thanks to this behavior and for the perfect weak scaling results in \cite{badia_highly_2014,badia_multilevel_2016}, the proposed BDDC methods (variants 1 and 2) are suitable for large-scale problems requiring fine meshes and a large number of subdomains. For Neumann boundary conditions, variant 1 is the best choice since it leads to a smaller coarse space than variant 2 for approximately the same number of PCG iterations. For Dirichlet boundary conditions with Nitsche's method, the 2nd variant is the preferred choice specially for fine meshes, since it leads to a better weak scaling.
\begin{figure}[ht!]
\centering
\includegraphics[width=0.9\textwidth]{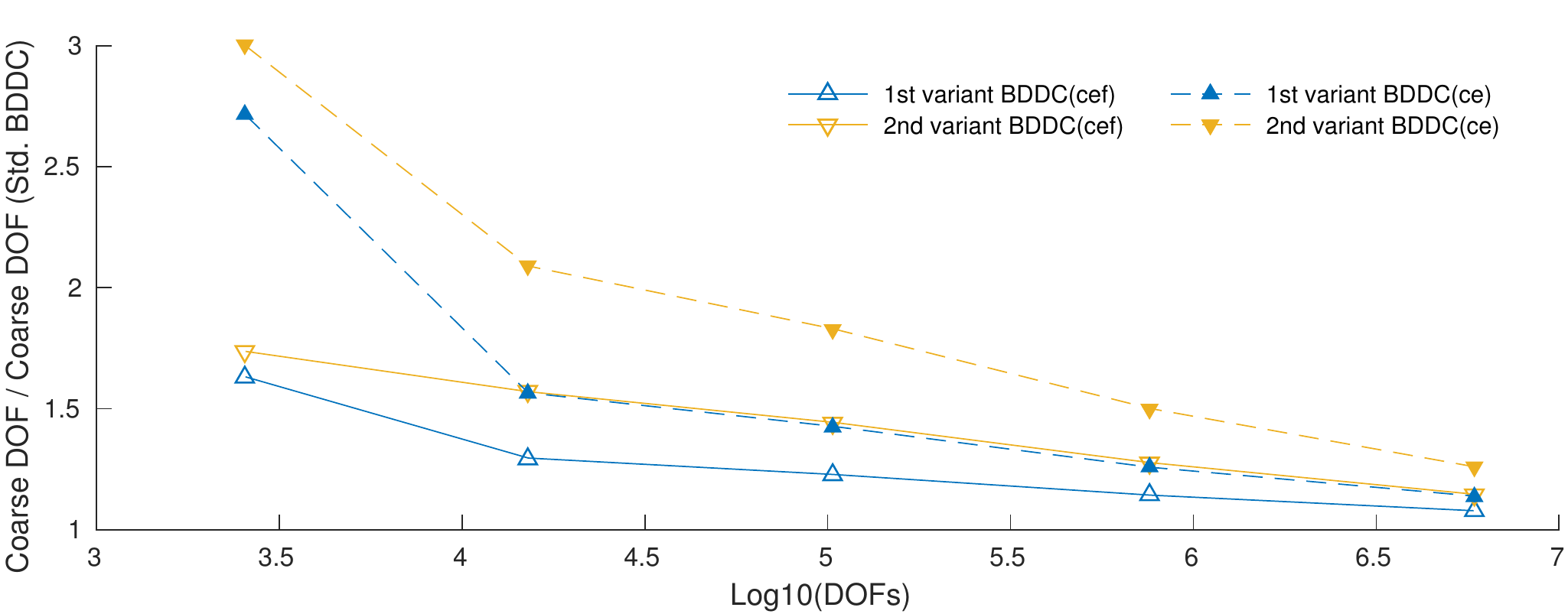}
\caption{\examplename: Size of the coarse space for the new BDDC methods (variants 1 and 2 in Table~\ref{table:methods}) with respect the size of the usual coarse space in BDDC.}
\label{fig:popcorn_wscal_cdof_cef_vs_ce}
\end{figure}

%
%\begin{figure}[ht!]
%\centering
%\begin{subfigure}{\textwidth}
%\centering
%\includegraphics[width=0.9\textwidth]{fig_popcorn_wscal_iters_nbc.pdf}
%\caption{Results for Neumann boundary conditions.}
%\end{subfigure}
%\begin{subfigure}{\textwidth}
%\centering
%\includegraphics[width=0.9\textwidth]{fig_popcorn_wscal_iters_dbc.pdf}
%\caption{Results for Dirichlet boundary conditions. }
%\end{subfigure}
%\caption{Weak scalability test for the ``Popcorn flake'' geometry displayed in  Fig.~\ref{fig:geometries}. {\color{red} TODO: compute full with Nitsche. now it is computed with strong DBC. Costa veure la linea groga al cas de Neumann}}
%\label{fig:popcorn_wscal_iters}
%\end{figure}

%
%\begin{figure}[ht!]
%\centering
%\begin{subfigure}{\textwidth}
%\centering
%\includegraphics[width=0.9\textwidth]{fig_popcorn_wscal_relcdofs_dbc.pdf}
%\caption{Results for Dirichlet boundary conditions.}
%\end{subfigure}
%\begin{subfigure}{\textwidth}
%\centering
%\includegraphics[width=0.9\textwidth]{fig_popcorn_wscal_relcdofs_nbc.pdf}
%\caption{Results for Neumann boundary conditions. {\color{red} Cal posar aquesta figura? no aporta info extra. Jo diria de paraula que només incloem el cas Dirichlet ja que el de Neumann és igual.}}
%\end{subfigure}
%\caption{ Size of the coarse space for the new BDDC methods (variants 1 and 2 in Table \TBD{} ) with respect the size of the usual coarse space in BDDC. The results are for the ``Popcorn'' geometry displayed in  Fig.~XX and are given for Dirichlet (Nitsche's method) and Neumann boundary conditions..}
%\end{figure}

{

  All the results presented so far are obtained considering a fixed ratio $H/h = 8$ between the characteristic sub-domain size $H$ and the characteristic mesh size $h$. We evaluate the effect of $H/h$ on the number of iterations of the three variants of the BDDC preconditioner for the popcorn flake geometry in Figures~\ref{fig:popcorn_wscal_iters_nloads} and~\ref{fig:popcorn_wscal_rcsize_nloads}. As expected, the effect of $H/h$ (measure of the local problem size) on the number of iterations is very mild. An algorithmic strong scalability analysis, i.e., to fix the global problem and consider $P$, $8P$, and $64P$ processors, can be obtained by looking for a given value of DOFs at the number of iterations for the three cases under consideration. These figures have not been included for the sake of brevity.
}

\begin{figure}[ht!]
\centering
\begin{subfigure}{\textwidth}
\centering
\includegraphics[width=0.9\textwidth]{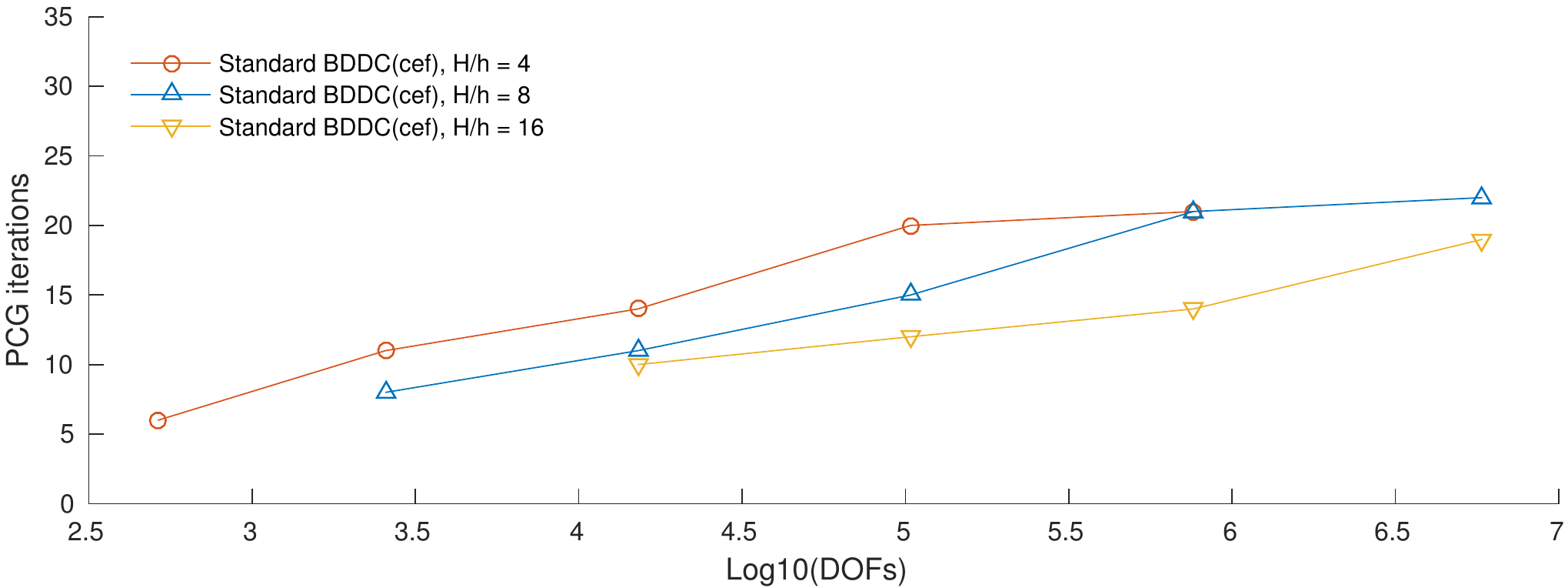}
\caption{ Results for standard coarse space.}
\label{fig:popcorn_wscal_iters_nloads:std}
\end{subfigure}
\begin{subfigure}{\textwidth}
\centering
\includegraphics[width=0.9\textwidth]{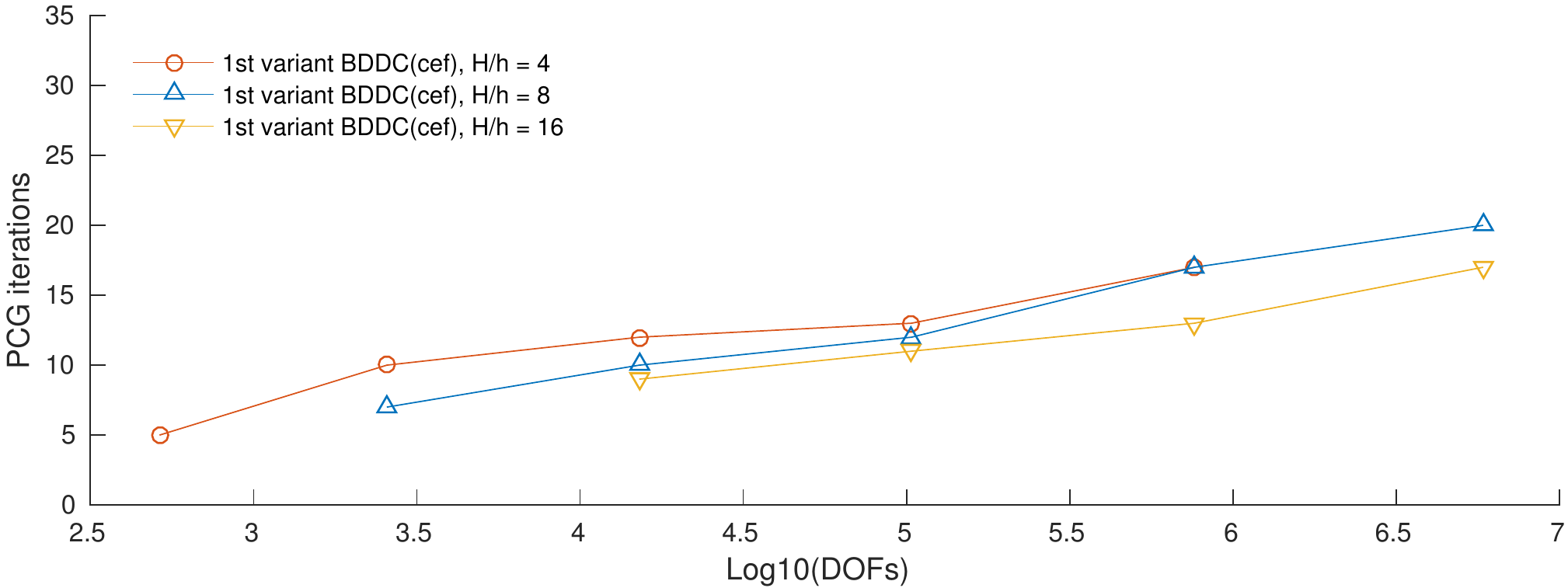}
\caption{ Results for 1st variant.}
\label{fig:popcorn_wscal_iters_nloads:v1}
\end{subfigure}
\begin{subfigure}{\textwidth}
\centering
\includegraphics[width=0.9\textwidth]{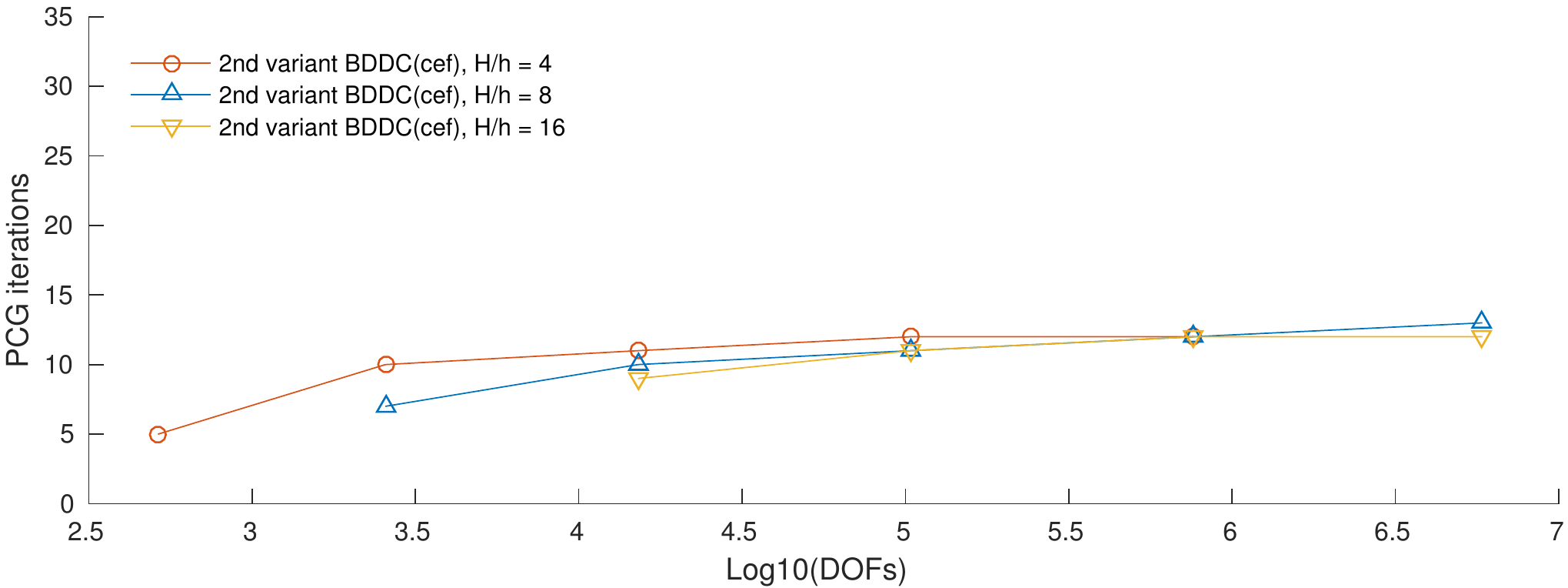}
\caption{ Results for 2nd variant.}
\label{fig:popcorn_wscal_iters_nloads:v2}
\end{subfigure}
\caption{ \examplename: Influence of the ratio $H/h$ in the linear solver weak scaling. Results for weak Dirichlet boundary conditions,  three different ratios, $H/h = 4$, $H/h = 8$ and $H/h = 16$, and for the BDDC methods ``Standard BDDC(cef)'', ``1st variant BDDC(cef)'', and ``2nd variant BDDC(cef)'' presented in Table~\ref{table:methods}.}
\label{fig:popcorn_wscal_iters_nloads}
\end{figure}

\begin{figure}[ht!]
\centering
\begin{subfigure}{\textwidth}
\centering
\includegraphics[width=0.9\textwidth]{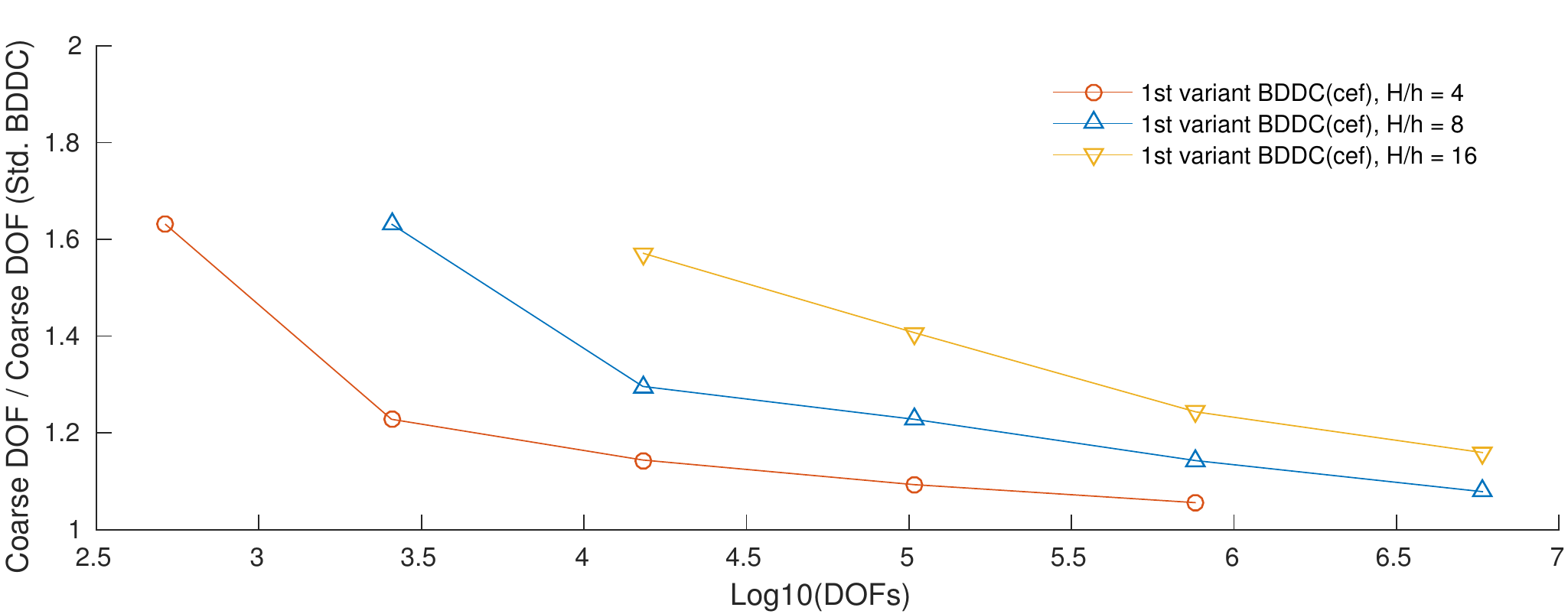}
\caption{ Results for 1st variant.}
\label{fig:popcorn_wscal_rcsize_nloads:v1}
\end{subfigure}
\begin{subfigure}{\textwidth}
\centering
\includegraphics[width=0.9\textwidth]{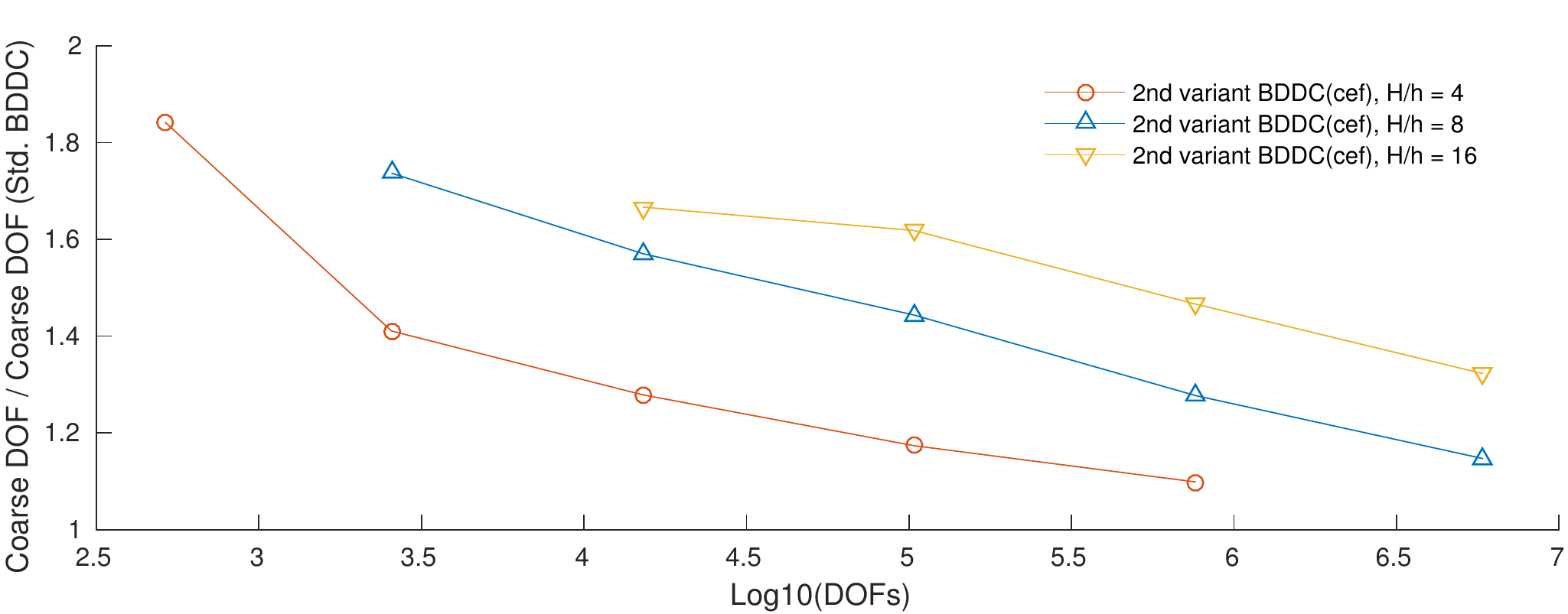}
\caption{ Results for 2nd variant.}
\label{fig:popcorn_wscal_rcsize_nloads:v2}
\end{subfigure}
\caption{ \examplename: Influence of the ratio $H/h$ in the coarse space size of the new BDDC methods (variants 1 and 2 in Table~\ref{table:methods}). Results for three different rations, $H/h = 4$, $H/h = 8$, and $H/h = 16$, and weak Dirichlet boundary conditions.}
\label{fig:popcorn_wscal_rcsize_nloads}
\end{figure}

It is also desirable that the corner detection mechanisms in variants 1 and 2 do not have a negative impact in the load balancing of the local operations involved in the BDDC preconditioner, namely the number of local Neumann problems required to compute the coarse basis functions, which is the main algorithmic source of load imbalance (the number of local Dirichlet problems and constrained Neumann problems in the application step are identical for all subdomains). The effect of the proposed method on the load balancing is studied in Fig.~\ref{fig:popcorn_histograms_coarse}. Here, the extra extra corners added by 1st and 2nd variants of the new method can potentially affect the load balancing. {However, the load imbalance in terms of the coarse DOFs per subdomain is not getting worse as the mesh is refined (see Fig~\ref{fig:popcorn_histograms_coarse}).} The results in Fig.~\ref{fig:popcorn_histograms_coarse}  are for the 2nd variant of the BDDC(ce) method. Even better results are obtained for the 1st variant and the BDDC(cef) method. The plots are not included for the sake of brevity since they do not add extra information.

\begin{figure}[ht!]
\centering
\begin{subfigure}{0.32\textwidth}
\centering
\includegraphics[width=\textwidth]{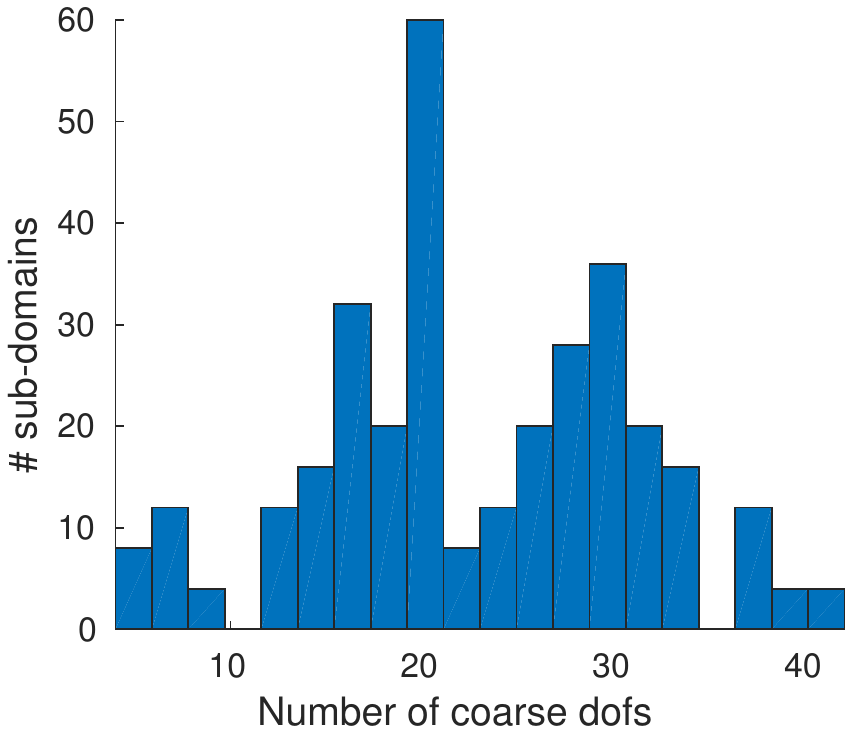}
\caption{Mesh 3.}
\end{subfigure}
\begin{subfigure}{0.32\textwidth}
\centering
\includegraphics[width=\textwidth]{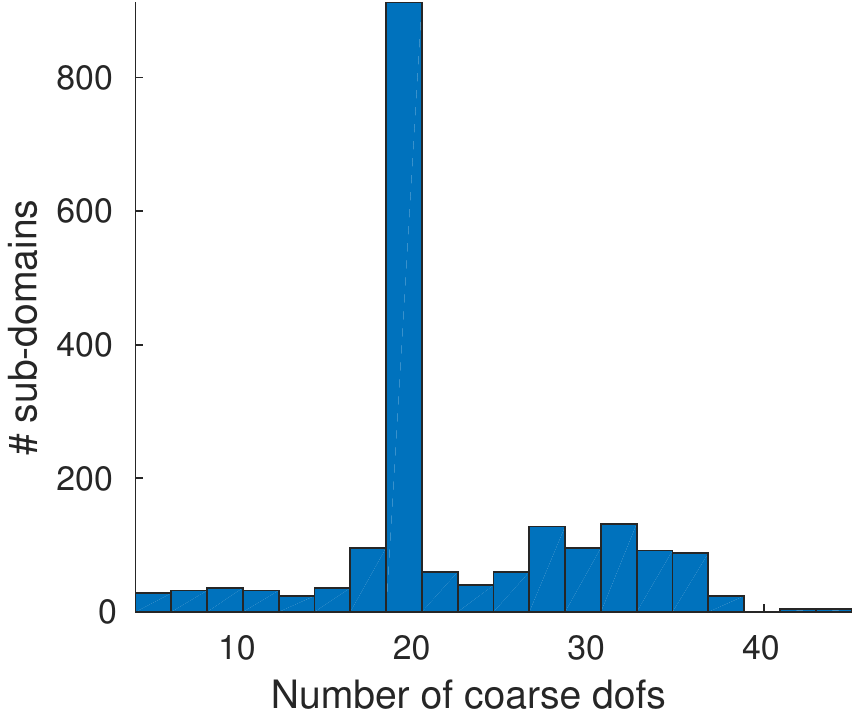}
\caption{Mesh 4.}
\end{subfigure}
\begin{subfigure}{0.32\textwidth}
\centering
\includegraphics[width=\textwidth]{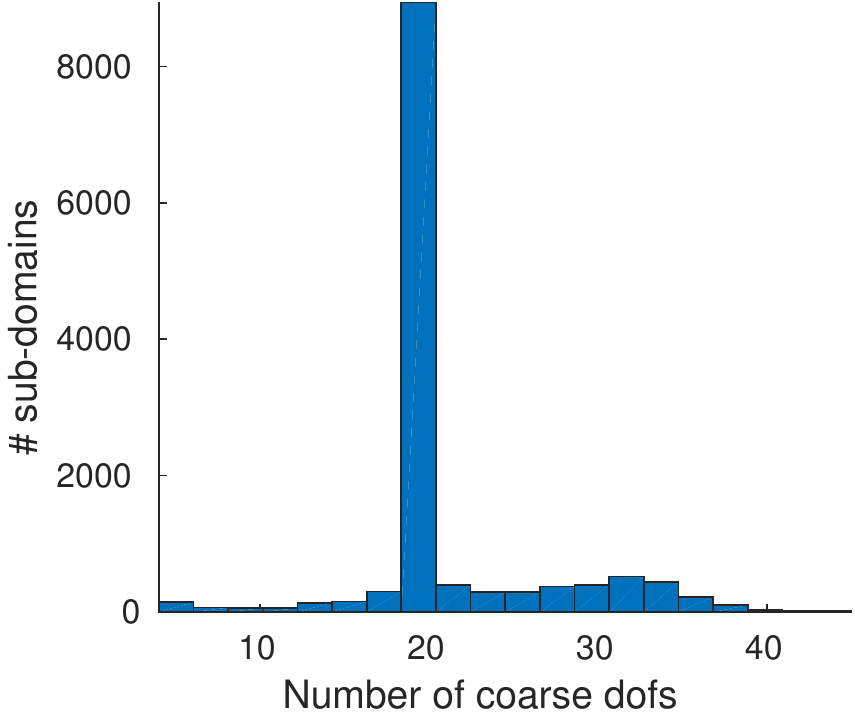}
\caption{Mesh 5.}
\end{subfigure}
\caption{\examplename: Distribution of the number of coarse basis functions across the subdomains.}
\label{fig:popcorn_histograms_coarse}
\end{figure}

%
%\begin{figure}[ht!]
%\centering
%\begin{subfigure}{0.32\textwidth}
%\centering
%\includegraphics[width=\textwidth]{fig_popcorn_histogram_diri.pdf}
%%\caption{}
%\end{subfigure}
%\begin{subfigure}{0.32\textwidth}
%\centering
%\includegraphics[width=\textwidth]{fig_popcorn_histogram_neum.pdf}
%%\caption{}
%\end{subfigure}
%\begin{subfigure}{0.32\textwidth}
%\centering
%\includegraphics[width=\textwidth]{fig_popcorn_histogram_coarse.pdf}
%%\caption{}
%\end{subfigure}
%\caption{\examplename: \color{red} Histograms for the popcorn, Dirichlet, bddc(ce), mesh 5. This is the most unfavorable case for the most relevant mesh. }
%\label{fig:popcorn_histograms}
%\end{figure}

In a final phase of the experiment, we study the robustness of the methods with respect to the position of the cut geometry. To this end, we slightly move the location of the background mesh by an arbitrary value $\varepsilon$ in one of the three spatial directions. By translating the mesh in this way, the location of the geometry changes and creates different kinds of intersections with the mesh elements, leading to cut elements with different active volume fractions that can affect the conditioning of the problem in different ways. Our goal is to show that the proposed methods are not affected by this translation and the different configurations of cut elements.  The values for $\varepsilon$ are chosen in the closed interval $[-h,+h]$, where $h$ is the element size of the background mesh. Fig.~\ref{fig:fig_popcorn_var_h} shows  the number of iterations of the PCG solver as a function of $\varepsilon$, for the mesh number 3 in Table~\ref{table:meshes}. The results are for weak Dirichlet boundary conditions, which is the most challenging case. The results in Fig.~\ref{fig:fig_popcorn_var_h} show that the 2nd variant is very robust with respect to the position of the interface. The maximum change in the number of iterations is only 1 iteration for all the 51 different values of $\varepsilon$ considered. The other methods have also a decent robustness with a maximum change of 4 iterations.   
%
%\begin{figure}[ht!]
%\centering
%\includegraphics[width=0.35\textwidth]{fig_popcorn_epsi.pdf}
%\caption{\examplename: Translation of the background mesh.}
%\label{fig:fig_popcorn_epsi}
%\end{figure}

\begin{figure}[ht!]
\centering
\includegraphics[width=0.9\textwidth]{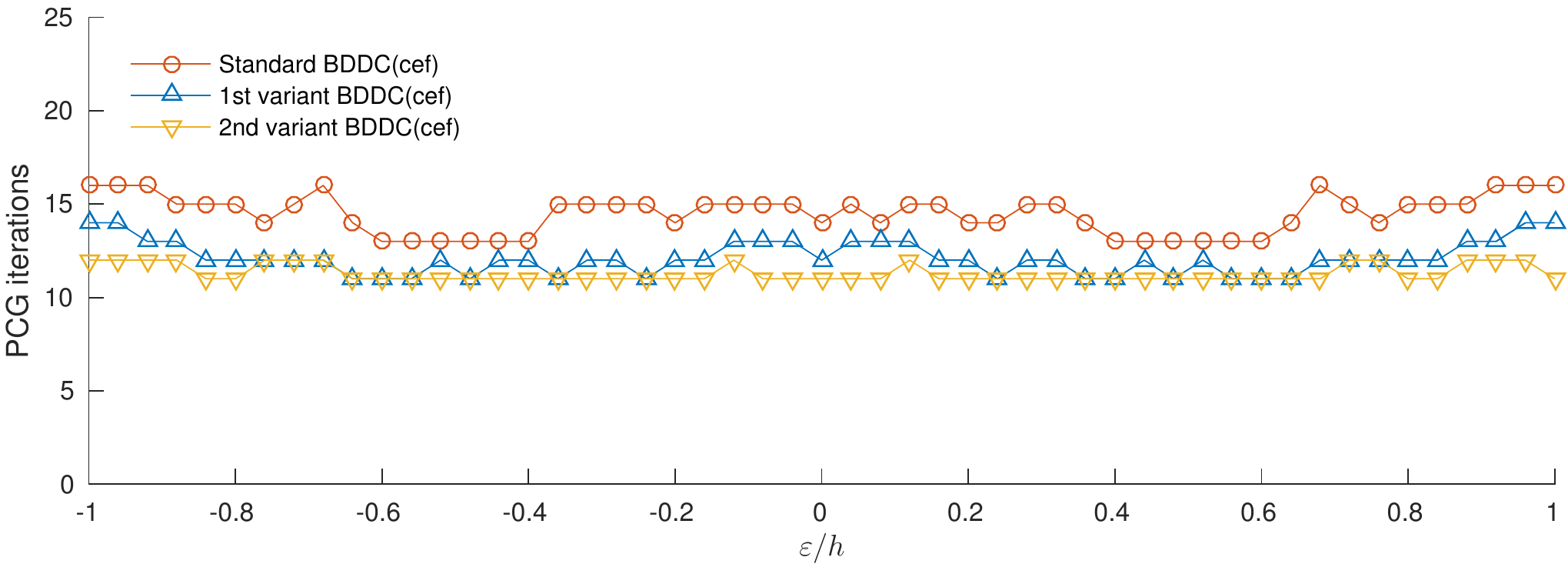}
\caption{\examplename: Study of the influence of the position of the boundary on the porposed methods.}
\label{fig:fig_popcorn_var_h}
\end{figure}

\renewcommand{\examplename}{Multi-body example}
\subsection{Multi-body example: sphere, popcorn flake, block, array of blocks and spiral.}
\label{sec:multi-body}

The purpose of this final example is to show that  the proposed BDDC preconditioners are able to deal with different complex geometries, and that the results obtained in previous example are qualitatively reproduced also in other settings. To this end, we study the BDDC preconditioners  for all the geometries previously presented in Fig.~\ref{fig:geometries}. The results of the weak scalability test are presented in Fig.~\ref{fig:fig_galery_wscal_iters}. The figure shows that performance of the methods for this set of geometries is similar to the one previously obtained for the ``popcorn flake''. For Neumann boundary conditions a perfect weak scaling is obtained for both the 1st and 2nd variants, whereas for Dirichlet boundary conditions a perfect weak scaling is only obtained with the 2nd variant. It is remarkable that the number of iterations is almost identical for all the geometries for the finest meshes (see 1st and 2nd variants for Neumann conditions, and 2nd variant for Dirichlet conditions). This is true even for the  array of blocks,  which is  the most complex geometry considered herein. 
\begin{figure}[ht!]
\centering
\begin{subfigure}{\textwidth}
\centering
\includegraphics[width=0.9\textwidth]{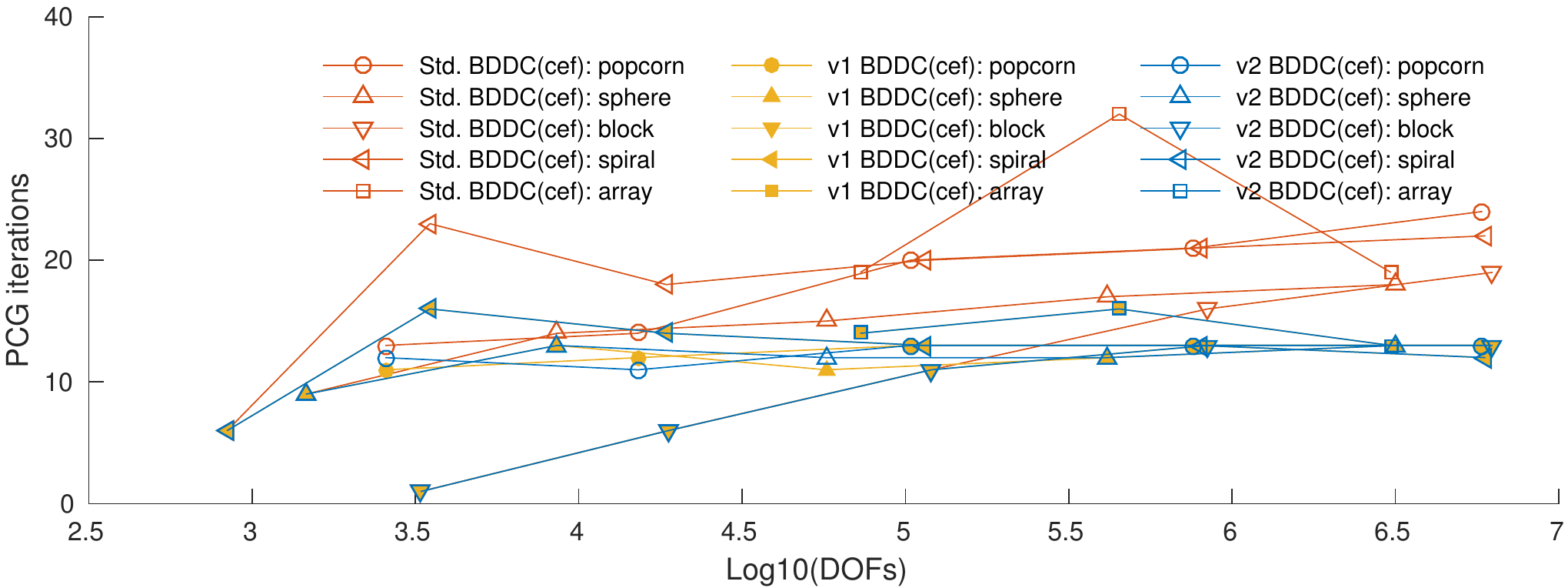}
\caption{Results for Neumann boundary conditions.}
\end{subfigure}
\begin{subfigure}{\textwidth}
\centering
\includegraphics[width=0.9\textwidth]{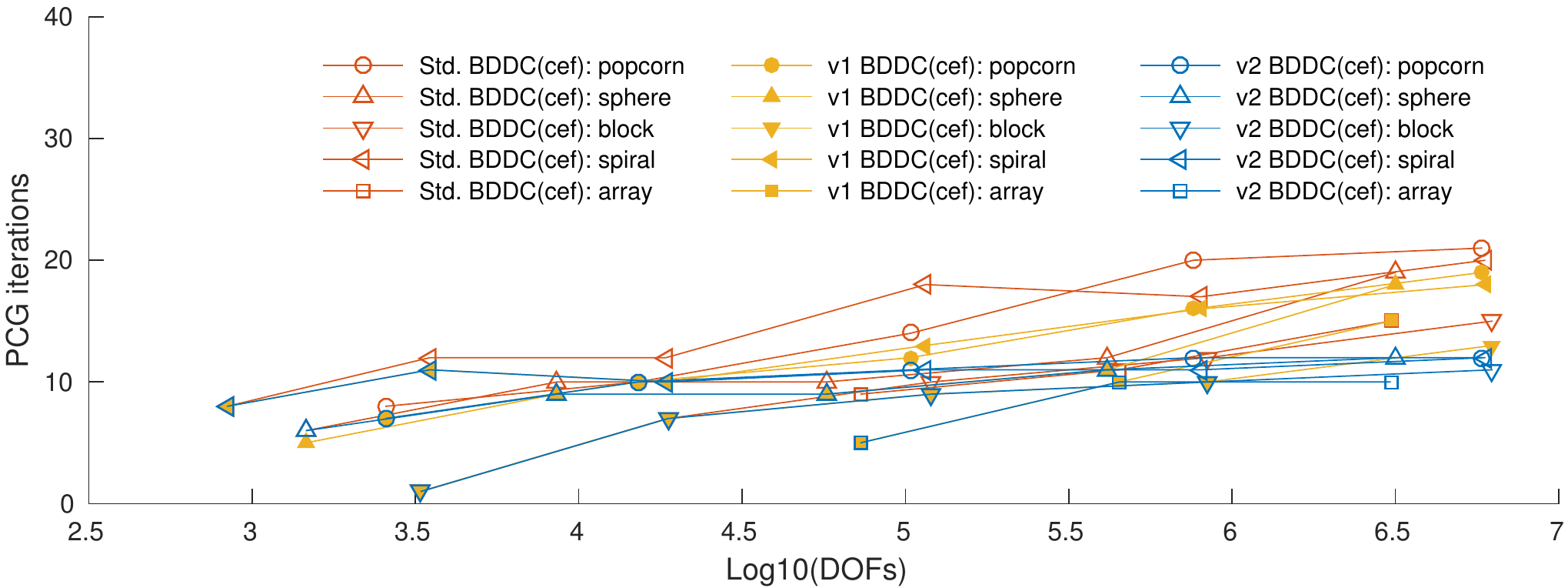}
\caption{Results for Dirichlet boundary conditions.}
\end{subfigure}
\caption{\examplename: Weak scalability test for all the different geometries displayed in Fig.~\ref{fig:geometries}.}
\label{fig:fig_galery_wscal_iters}
\end{figure}

%\begin{figure}[ht!]
%\centering
%\begin{subfigure}{\textwidth}
%\centering
%\includegraphics[width=0.9\textwidth]{fig_galery_wscal_iters_dbc_ce.pdf}
%\caption{Results Dirichlet boundary conditions.}
%\end{subfigure}
%\begin{subfigure}{\textwidth}
%\centering
%\includegraphics[width=0.9\textwidth]{fig_galery_wscal_iters_nbc_ce.pdf}
%\caption{Results for Neumann boundary conditions.}
%\end{subfigure}
%\caption{Weak scalability test for all the different geometries displayed in  Fig.~XX. The results are given for Dirichlet (Nitsche's method) and Neumann boundary conditions. The considered inter-domain continuity is of type BDDC(ce). {\color{red} TODO: encara falta calcular alguns casos. Si al calcular tots els casos, la figura queda molt igual a la de BDDC(cef), jo potser no la posaria. Diria de paraula que BDDC(ce) no està inclòs ja que els resultats son molt semblants a BDDC(cef).}}
%\end{figure}

The size of the coarse spaces behaves similar to the previous example (see Fig.~\ref{fig:galery_relcdofs}). The increment in the coarse space size associated with variants 1st and 2nd becomes smaller in front of the size of the coarse space in standard BDDC as the mesh is refined. This is again due to the fact that the number of objects in the bulk of the domain scales faster than the objects cut by the boundary. For that reason, the sphere is the shape that  presents a smaller increment in the coarse space size. It minimizes the cut surface per unit of domain volume. Note, however, that the coarse space size of the variants 1 and 2  tends also to the coarse space size of standard BDDC for the other shapes (see, e.g., the spiral or the array of blocks).
\begin{figure}[ht!]
\centering
\includegraphics[width=0.9\textwidth]{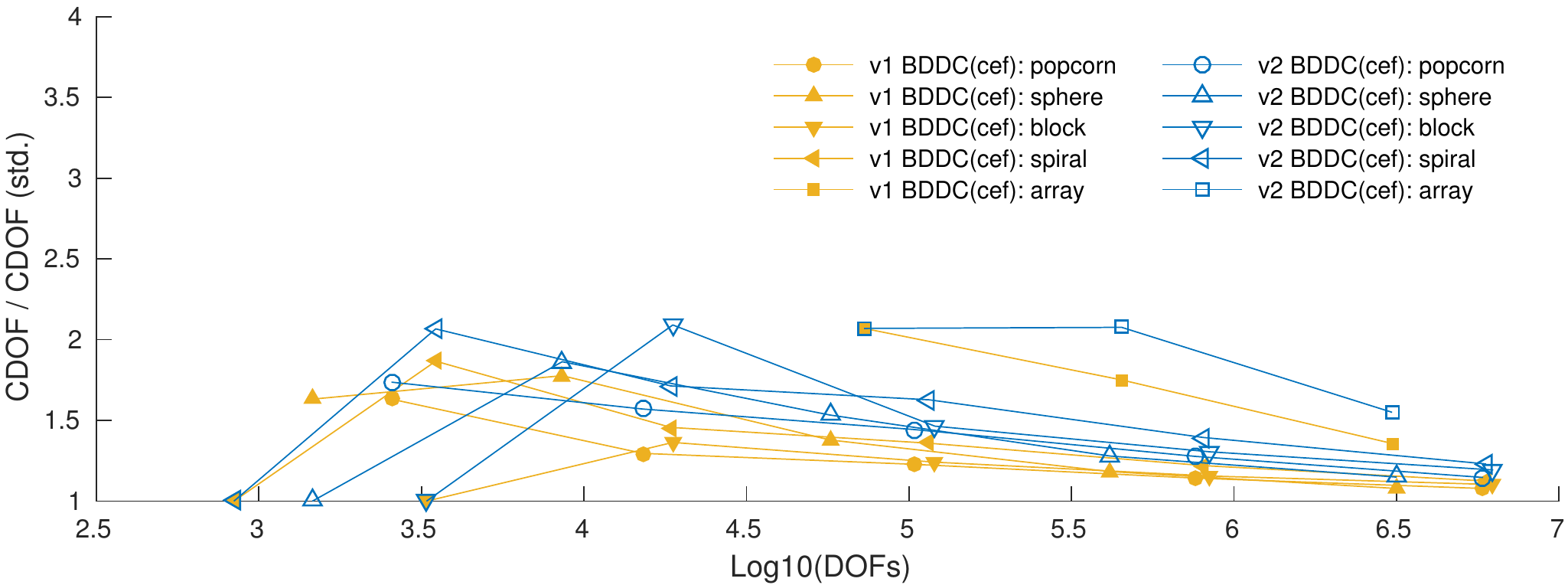}
\caption{\examplename: Size of the coarse space for the new BDDC methods (variants 1 and 2 in Table~\ref{table:methods}) with respect to the size of the usual coarse space in BDDC.}
\label{fig:galery_relcdofs}
\end{figure}

%\begin{figure}[ht!]
%\centering
%\begin{subfigure}{\textwidth}
%\centering
%\includegraphics[width=0.9\textwidth]{fig_galery_wscal_relcdofs_dbc_cef.pdf}
%\caption{Results for Dirichlet boundary conditions.}
%\end{subfigure}
%\begin{subfigure}{\textwidth}
%\centering
%\includegraphics[width=0.9\textwidth]{fig_galery_wscal_relcdofs_nbc_cef.pdf}
%\caption{Results for Neumann boundary conditions. {\color{red} Aprota algo?}}
%\end{subfigure}
%\caption{Size of the coarse space for the new BDDC methods (variants 1 and 2 in Table \TBD{} ) with respect the size of the usual coarse space in BDDC. The results are for all the different geometries displayed and for Dirichlet (Nitsche's method) and Neumann boundary conditions. The considered inter-domain continuity is of type BDDC(cef).}
%\end{figure}

%\begin{figure}[ht!]
%\centering
%\begin{subfigure}{\textwidth}
%\centering
%\includegraphics[width=0.9\textwidth]{fig_galery_wscal_relcdofs_dbc_ce.pdf}
%\caption{Results for Dirichlet boundary conditions.}
%\end{subfigure}
%\begin{subfigure}{\textwidth}
%\centering
%\includegraphics[width=0.9\textwidth]{fig_galery_wscal_relcdofs_nbc_ce.pdf}
%\caption{Results for Neumann boundary conditions. {\color{red} Aprota algo?}}
%\end{subfigure}
%\caption{Size of the coarse space for the new BDDC methods (variants 1 and 2 in Table \TBD{} ) with respect the size of the usual coarse space in BDDC. The results are for all the different geometries displayed and for Dirichlet (Nitsche's method) and Neumann boundary conditions. The considered inter-domain continuity is of type BDDC(ce). }
%\end{figure}

Finally, we study the robustness of the methods with respect to the position of the unfitted boundary $\partial\Omega$. In all cases, we move the background mesh a value $\varepsilon\in[-h,+h]$ to see  the influence of the boundary location. Fig.~\ref{fig:galery_var_h} shows the number of iterations of the PCG solver with respect to the value $\varepsilon$ for all the shapes. The results correspond to weak Dirichlet boundary conditions, which is the more challenging case. The results are similar to previous example. It is observed that the performance of the 2nd variant is very robust. It is almost independent of the position of the interface with a maximum change of 1 iteration in all the  considered geometries. 
\begin{figure}[ht!]
\centering
\includegraphics[width=0.9\textwidth]{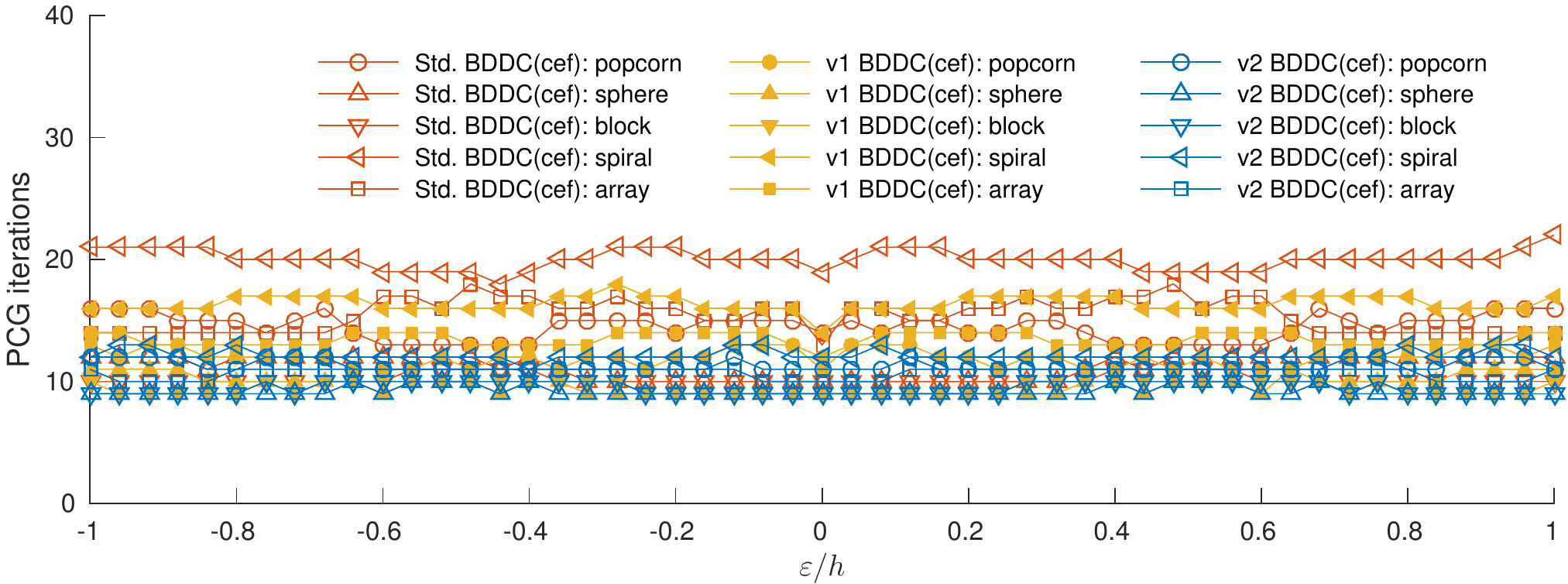}
\caption{\examplename: Study of the influence of the position of the boundary on the porposed methods.}
\label{fig:galery_var_h}
\end{figure}

In summary, the presented results show that the proposed variants of the BDDC method are able to provide excellent results for different and complex geometries discretized with unfitted FE meshes. As in previous example, the 1st variant is already enough to achieve a perfect weak scaling for all the shapes for the case of Neumann boundary conditions. The 2nd variant is the only one that leads to perfect weak scaling for Dirichlet conditions. It is remarkable that the number of PCG iterations is nearly independent from the geometry for the finest meshes. This shows that the methods are able to effectively get rid of the complexities introduced by the considered geometries and cut elements. The size of the coarse spaces of the new methods tends to  the size for standard BDDC as the mesh is refined.  Therefore, the efficiency of the new methods is comparable to the one of standard BDDC for fine meshes.

%{\color{red} TODO posar les figures de var h. Posar tb neumann}.

\section{Conclusions}\label{conclusions}

In this work, we have analyzed the main problems related to the use of BDDC methods to unfitted FE meshes. From a mathematical point of view, there are different assumptions in the analysis of these algorithms that do not hold anymore for unfitted meshes. As a result, the polylogarithmic condition number bounds in BDDC methods are not true in general. In fact, using a breakdown example, we have observed that the condition number of these methods cannot be robust with respect to the intersection between the surface and the background mesh. One can { fulfill again the assumptions of the BDDC analysis} by considering as corner constraint all the DOFs on the interface that belong to cut cells. Unfortunately, even thought the resulting algorithm is perfectly robust, such approach would lead to a dramatic increase in the coarse problem size, and it is not considered in practice.

Alternatively, and motivated by both numerical experimentation and the mathematical analysis of BDDC methods, we have considered a stiffness-based weighting and a subdivision of the BDDC edges. The first ingredient has turned to be an essential ingredient to get a robust BDDC solver for unfitted meshes. { On the other hand, the}  enrichment of the coarse BDDC space is especially important in regions with weak imposition of Dirichlet boundary conditions. As usual, the use of stiffness weighting is not backed up with a full mathematical theory. Instead, a complete set of numerical experiments has been performed to show the robustness and algorithmic weak scalability of the proposed solvers.

Future work will include the extension to more complex (non-coercive) physical problems and to ghost penalty stabilization. The use of stabilized versions is expected to be easier than the problem considered herein, since the condition number bound in these cases can be bounded independently of the location of the intersection. Furthermore, we want to implement these solvers within the \texttt{FEMPAR} library \cite{fempar-manual,fempar}, in order to show weak scalability results for its highly scalable implementation~\cite{badia_implementation_2013,badia_multilevel_2016} of multilevel BDDC preconditioners. We also want to apply these methods to adaptive octree meshes and space-filling curves \citep{bader_space-filling_2012}.

\bibliographystyle{siam} \bibliography{art021.bib}

\end{document}